\documentclass[11pt]{amsart}

\usepackage{color,graphicx,amssymb,latexsym,amsfonts,txfonts,amsmath,amsthm}
\usepackage{pdfsync}
\usepackage{amsmath,amscd}
\usepackage[all,cmtip]{xy}

\usepackage{hyperref}
\hypersetup{
    colorlinks=true,       
    linkcolor=blue,          
    citecolor=blue,        
    filecolor=blue,      
    urlcolor=blue           
}

\advance\hoffset by -0.9 truecm

\newcommand{\s}{\vspace{0.3cm}}

\newcommand{\be}{\begin{equation}}
\newcommand{\ee}{\end{equation}}
\numberwithin{equation}{section}

\begin{document}

\newtheorem{theorem}{Theorem}[section]                  
\newtheorem{proposition}[theorem]{Proposition}                  
\newtheorem{corollary}[theorem]{Corollary}
\newtheorem{lemma}[theorem]{Lemma}
\newtheorem{conjecture}[theorem]{Conjecture} 
\newtheorem{definition}[theorem]{Definition}
\theoremstyle{remark}         
\newtheorem{remark}[theorem]{\bf Remark}
\newtheorem{example}[theorem]{\bf Example}
\newtheorem{question}[theorem]{\bf Question}

\title{Isogenous decomposition of the Jacobian of generalized Fermat curves}

\author[M. Carvacho, R. A. Hidalgo and S. Quispe]{Mariela Carvacho, Rub\'en A. Hidalgo and Sa\'ul Quispe}

\address{Departamento de Matem\'aticas, Universidad T\'ecnica Federico Santa Mar\'{\i}a, Valpara\'{\i}so, Chile}
\email{mariela.carvacho@usm.cl}

\address{Departamento de Matem\'atica, Universidad de la Frontera, Casilla 54-D, Temuco, Chile}
\email{ruben.hidalgo@ufrontera.cl, saul.quispe@ufrontera.cl}

\thanks{Partially supported by projects Fondecyt 1150003 and 3140050}

\begin{abstract}
 A closed Riemann surface $S$ is called a generalized Fermat curve of type $(p,n)$, where $p,n \geq 2$ are integers, if it admits a group $H \cong {\mathbb Z}_{p}^{n}$ of conformal automorphisms so that $S/H$ is an orbifold of genus zero with exactly $n+1$ cone points, each one of order $p$. It is known that $S$ is a fiber product of $(n-1)$ classical Fermat curves of degree $p$ and, for $(p-1)(n-1)>2$, that it is a non-hyperelliptic Riemann surface. In this paper, assuming $p$ to be a prime integer, we  provide a decomposition, up to isogeny, of the Jacobian variety $JS$ as a product of Jacobian varieties of certain cyclic $p$-gonal curves.  Explicit equations for these $p$-gonal curves are provided in terms of the equations for $S$. As a consequence of this decomposition, we are able to provide explicit positive-dimensional families of closed Riemann surfaces whose Jacobian variety is isogenous to the product of elliptic curves.

\end{abstract}

\maketitle

\section{Introduction}
A principally polarized abelian variety $A$ is simple if there are no lower dimensional principally polarized abelian varieties $B$ and $C$ so that $A$ is isogenous to the product $B \times C$. In the case that $A$ is non-simple, Poincar\'e's complete reducibility theorem \cite{Poincare} asserts the existence of simple principally polarized abelian varieties $A_{1}, \ldots, A_{s}$ and positive integers $n_{1},\ldots,n_{s}$ such that $A$ is isogenous to the product $A_{1}^{n_{1}} \times \cdots A_{s}^{n_{s}}$. Moreover, the $A_{j}$ and $n_{j}$ are unique (up to isogeny) and permutation of the factors (see Section \ref{sec:poincare}). 

A very interesting class of principally polarized abelian varieties are the Jacobian variety $JS$ of a closed Riemann surface $S$ of genus $g \geq 1$ (see Section \ref{sec:jacobian}).  We say that $JS$ is Jacobian-simple if it is non-isogenous to a non-trivial product of Jacobian varieties of of lower genera surfaces; otherwise  we say that it is Jacobian-non-simple. It may be that $JS$ is Jacobian-simple but non-simple as a polarized abelian variety; of course, if it is simple then it is also Jacobian-simple. Given $S$, one would like to know if it is Jacobian-simple or not. It is clear that $JS$ is Jacobian-simple for $g=1$. If $g=2$, then $JS$ is simple if and only if it is Jacobian-simple. As every principally polarized abelian variety of dimension two is either a Jacobian or isogenous to the product of two elliptics, then for $g=3$ again $JS$ is simple if and only if it is Jacobian-simple. The two properties may differ for $g \geq 4$.

In this paper we are interested in decomposing, up to isogeny, the Jacobian variety of a closed Riemann surface as product of Jacobian varieties of lower genera Riemann surfaces (this is a first approach to finding a Jacobian decomposition of a Jacobian variety). 

With respect to the above, Ekedahl and Serre \cite{E-S} constructed examples of closed Riemann surfaces $S$ for which $JS$ is isogenous to the product of elliptic curves. These examples are for genus $g\leq 1297$ with some gaps. More examples of Riemann surfaces $S$ for which $JS$ is isogenous to the product of elliptic curves have been produced by many authors, see for instance \cite{Shaska2,Nakajima,Shaska1,Yamauchi1,Yamauchi2}. Two fundamental question, stated by Ekedahl and Serre in \cite{E-S} are the following. {\it Q1} Is there a Riemann surface $S$ in every genus $g \geq 2$ such that $JS$ is isogenous to the product of elliptic curves?
{\it Q2} Is there a bound on the genus $g$ with the above decomposition property?. To the actual acknowledgment of the authors, both of these questions are still open.

In the case of hyperelliptic Riemann surfaces with a conformal automorphism of order two, different from the hyperelliptic involution, it is known that the Jacobian variety decomposes, up to isogeny, as the product of two Jacobian varieties. We recall this situation in Section \ref{sec:hiper} and we provide a two-dimensional family of genus four hyperelliptic Riemann surfaces whose Jacobian variety is isogenous to the product of four elliptic curves (explicit equations are provided).

Next, we restrict to a certain interesting families of Riemann surfaces called generalized Fermat curves of type $(p,n)$, where $p,n \geq 2$ are integers (see Section \ref{Sec:GFC} for details). These surfaces $S$ can be described as a suitable fiber product of $(n-1)$ classical Fermat curves of degree $p$ and, for $(n-1)(p-1)>2$, they are non-hyperelliptic and they admit a unique group of automorphisms $H \cong {\mathbb Z}_{p}^{n}$ with $S/$ of genus zero and exactly $(n+1)$ cone points; each one of order $p$.

 Our main result is Theorem \ref{(p,n)} which, in the case that $p$ is a prime integer, provides an isogenous decomposition of the Jacobian variety of generalized Fermat curves of type $(p,n)$ as a product of Jacobians of certain cyclic $p$-gonal curves. The equations of these $p$-gonal curves are given explicitly in terms of the equations of the corresponding generalized Fermat curve. Some consequences of the isogenous Jacobian decomposition are the following.
\begin{enumerate}
\item The Jacobian variety of any closed Riemann surface of genus $5$ admitting ${\mathbb Z}_{2}^{4}$ as a group of conformal automorphisms (a two-dimensional family) is isogenous to the product of elliptic curves. 

\item We obtain a one-dimensional family of closed Riemann surfaces of genus $17$ whose Jacobian varieties are isogenous to the product of elliptic curves.

\item We provide an example of a closed Riemann surface of genus $49$ whose Jacobian variety is isogenous to the product of elliptic curves.

\item The Jacobian variety of every closed Riemann surface of genus $10$ admitting a group $H \cong {\mathbb Z}_{3}^{3}$ as a group of conformal automorphisms whose quotient orbifold $S/H$ is the Riemann sphere with exactly three cone points, each one of order three, (a one-dimensional family) is isogenous to the product of elliptic curves.

\item The Jacobian variety of every closed Riemann surface of genus $65$ admitting a group $H \cong {\mathbb Z}_{3}^{4}$ as a group of conformal automorphisms with $S/H$ being of genus zero and having exactly five cone points, each one of order five, (a two-dimensional family) is isogenous to the product of $40$ elliptic curves and five $5$-dimensional Jacobians. 

\end{enumerate}

The facts (1) and (2), for the particular case that $p=2$ and the curves are defined over $\overline{\mathbb Q}$, were obtained by Yamauchi in \cite{Yamauchi1} (in our result we have not such an arithmetical restriction). Our results are also related to those obtained by Shaska in \cite{Shaska2}.  

Finally, we should mention the survey article by Rodr\'{\i}guez \cite{Rod:survey} in which the author review part of the theory of abelian varieties with group actions and the decomposition of them up to isogeny. Somewhere else we will try to describe our decomposition results for generalized Fermat curves in their language.

\section{Preliminaries}
In this section we recall some definitions and previous results we will need in the rest of the paper.

If $S$ denotes a closed Riemann surface, then ${\rm Aut}(S)$ will denote its full group of conformal automorphisms. If $H<{\rm Aut}(S)$, then we denote by ${\rm Aut}_{H}(S)$ the normalizer of $H$ inside ${\rm Aut}(S)$, that is, the biggest subgroup of ${\rm Aut}(S)$ containing $H$ as normal subgroup. 

\s
\subsection{Generalized Fermat curves}\label{Sec:GFC}
Let $S$ be a generalized Fermat curve of type $(p,n)$. A group $H<{\rm Aut}(S)$, $H\cong {\mathbb Z}_{p}^{n}$, so that the quotient orbifold $S/H$ has genus zero and exactly $n+1$ cone points of order $p$, is called a generalized Fermat group of type $(p,n)$ for $S$; the pair $(S,H)$ is called a generalized Fermat pair of type $(p,n)$. As a consequence of the Riemnn-Hurwitz formula, a generalized Fermat curve of type $(p,n)$ has genus
$$
g=g_{p,n}=1+\dfrac{\phi(p,n)}{2}, \quad \mbox{where} \quad \phi(p,n)=p^{n-1}((n-1)p-n-1),
$$

We say that two generalized Fermat pairs $(S_{1},H_{1})$ and $(S_{2},H_{2})$ are topologically (holomorphically) equivalent if there is some orientation-preserving homeomorphism (holomorphic homeomorphism) $f:S_{1} \to S_{2}$ so that $f H_{1} f^{-1}=H_{2}$. 

\s
\noindent
\begin{theorem}[\cite{GHL}]\label{N-H}\label{unif}
Any two generalized Fermat pairs of type $(p,n)$ are topologically equivalent. 
\end{theorem}

\s
\noindent
\begin{theorem}[\cite{GHL}]
If $p$ is a prime integer, then a generalized Fermat curve of type $(p,n)$ has a unique, up to conjugation, generalized Fermat group of the same type $(p,n)$.
\end{theorem}

\s
\noindent
\begin{theorem}[\cite{GHL}]\label{cor2}
Let $(S,H)$ be a generalized Fermat pair of type $(p,n)$ and let $P:S \to S/H$ be a branched regular covering with $H$ as group of cover transformations.
\begin{itemize}
\item[1.-] If ${\rm Aut}_{H}(S)$ denotes the normalizer of $H$ inside ${\rm Aut}(S)$, then each orbifold automorphism of $S/H$ lifts to an automorphism in ${\rm Aut}_{H}(S)$; that is, for each orbifold automorphism $\tau:S/H \to S/H$ there is a conformal automorphism $\widehat{\tau}:S \to S$ so that $P \widehat{\tau}=\tau P$.

\item[2.-] There exist elements of order $p$ in $H$, say $a_{1},\ldots,a_{n}$, so that the following hold. 
\begin{itemize}
\item[(i)] $H=\langle a_{1},\ldots,a_{n}\rangle$.
\item[(ii)] Each $a_{1},\ldots,a_{n}$ and $a_{n+1}=a_{1}a_{2} \cdots a_{n}$ has exactly $p^{n-1}$ fixed points.
\item[(iii)] If $h \in H$ has fixed points, then $h \in \langle a_{1} \rangle \cup \cdots \langle a_{n} \rangle \cup \langle a_{n+1} \rangle$.
\item[(iv)] If $h \in H$ is an element of order $p$ with fixed points and $x,y$ are any two of these fixed points, then  there is some $h^{*} \in H$ so that $h^{*}(x)=y$.
\item[(v)] If $p$ is a prime integer and $h \in H-\{I\}$ has no fixed points, then no non-trivial power of $h$ has fixed points.
\end{itemize}

Such a set of generators $a_{1}$,..., $a_{n}$ shall be called a {\it standard set of generators} for the generalized Fermat group $H$. 

\end{itemize}
\end{theorem}

\s
\noindent
\begin{remark}\label{rema:unico}
In \cite{HKLP} it has been proved that every generalized Fermat curve $S$ of type $(p,n)$, where $(n-1)(p-1)>2$, has a unique generalized Fermat group $H$ of that type. In particular, ${\rm Aut}_{H}(S)={\rm Aut}(S)$.
\end{remark}

\s
\subsubsection{Non-hyperbolic generalized Fermat curves}
The generalized Fermat curves of genus $g<2$ are  the following ones.
\begin{itemize}
\item[(i)] $(p,n)=(2,2)$:  $S=\widehat{\mathbb C}$ and $H=\langle A(z)=-z, B(z)=1/z\rangle$. 

\item[(ii)] $(p,n)=(3,2)$: $S={\mathbb C}/\Lambda_{e^{2 \pi i/3}}$, where
$\Lambda_{e^{2 \pi i/3}}=\langle A(z)=z+1, B(z)=z+e^{2 \pi i/3}\rangle$, and $H$ is generated by the induced transformations of $R(z)=e^{2 \pi i/3}z$ and $T(z)=z+(2+e^{2 \pi i/3})/3$. In this case, the $3$ cyclic groups $\langle R\rangle$, $\langle TR\rangle$ and $\langle T^{2}R \rangle$ project to the only $3$ cyclic subgroups in $H$ with fixed points; $3$ fixed points each one. This is provided by the degree $3$  Fermat curve $x^{3}+y^{3}+z^{3}=0$.

\item[(iii)] $(p,n)=(2,3)$: $S={\mathbb C}/\Lambda_{\tau}$, where $\tau \in {\mathbb H}^{2}=\{ z \in {\mathbb C}: {\rm Im}(\tau)>0\}$,  $\Lambda_{\tau}=\langle A(z)=z+1, B(z)=z+\tau\rangle$, and $H$ is generated by the induced transformations from $T_{1}(z)=-z$, $T_{2}(z)=-z+1/2$ and $T_{3}(z)=-z+\tau/2$. In this case, the conformal involutions induced on the torus by $T_1$, $T_2$, $T_3$ and their product are the only ones acting with fixed points; $4$ fixed points each. This is also described by the algebraic curve 
$\{x^{2}+y^{2}+z^{2}=0, \; \lambda x^{2}+y^{2}+w^{2}=0\}$, where $\lambda \in {\mathbb C}-\{0,1\}$.
\end{itemize}

\s
\subsubsection{Hyperbolic generalized Fermat curves}
A generalized Fermat curve of type $(p,n)$ has genus bigger than one if and only if $(n-1)(p-1)>2$; we say that the generalized Fermat pair (curve) is {\it hyperbolic}. 

Next result asserts that hyperbolic generalized Fermat curves are non-hyperelliptic (but Remark \ref{rema:unico} asserts that they behave somehow as if they were hyperelliptic ones).

\s
\noindent
\begin{theorem}[\cite{GHL}]\label{teo1}
A hyperbolic generalized Fermat curve is non-hyperelliptic.
\end{theorem}

\s
\subsubsection{Algebraic description}
Let $(S,H)$ be a generalized Fermat pair of type $(p,n)$ and,  up to a M\"obius transformation, let
$\{\infty,0,1,\lambda_{1},\lambda_{2},\ldots,\lambda_{n-2}\}$
be the conical points of $S/H$.

Let us consider the following fiber product of $(n-1)$ classical Fermat curves
$$
C^{p}(\lambda_{1},\ldots,\lambda_{n-2}): \left\{ \begin{array}{ccc}
x_{1}^{p}+x_{2}^{p}+x_{3}^{p}&=&0\\
\lambda_{1}x_{1}^{p}+x_{2}^{p}+x_{4}^{p}&=&0\\
\lambda_{2}x_{1}^{p}+x_{2}^{p}+x_{5}^{p}&=&0\\
\vdots & \vdots& \vdots\\
\lambda_{n-2}x_{1}^{p}+x_{2}^{p}+x_{n+1}^{p}&=&0
\end{array}
\right\} \subset {\mathbb P}^{n}.
$$

 The conditions on the parameters $\lambda_{j}$ ensure that $C^{p}(\lambda_{1},\ldots,\lambda_{n-2})$ is a non-singular projective algebraic curve, that is, a closed Riemann surface. On $C^{p}(\lambda_{1},\ldots,\lambda_{n-2})$ we have the Abelian group $H_{0} \cong {\mathbb Z}_{p}^{n}$ of conformal automorphisms generated by the transformations 
$$a_{j}([x_{1}:\cdots:x_{n+1}])=[x_{1}:\cdots:x_{j-1}: \omega_{p} x_{j}:x_{j+1}:\cdots:x_{n+1}], \; j=1,...,n,$$ 
where $\omega_{p}=e^{2 \pi i/p}$.

If we consider the degree $p^{n}$ holomorphic map 
$$\pi:C^{p}(\lambda_{1},...,\lambda_{n-2}) \to \widehat{\mathbb C}$$ given by
$$\pi([x_{1}:\cdots:x_{n+1}])=-\left( \dfrac{x_{2}}{x_{1}} \right)^{p},$$
then $\pi \circ a_{j} = \pi$, for every $a_{j}$, $j=1,\ldots,n$. It follows that $C^{p}(\lambda_{1},\ldots,\lambda_{n-2})$ is a generalized Fermat curve with generalized Fermat group $H_0$; whose standard generators are $a_{1},\ldots, a_{n}$ and $a_{n+1}=a_{1}a_{2}\cdots a_{n}$. 

The fixed points of $a_{j}$ on $C^{p}(\lambda_{1},\ldots,\lambda_{n-2})$ are given by the intersection 
$$Fix(a_{j})=\{x_{j}=0\} \cap C^{p}(\lambda_{1},\ldots,\lambda_{n-2}).$$ 

In this way, the branch values of $\pi$ are given by the points
$$\pi(Fix(a_{1}))=\infty, \; \pi(Fix(a_{2}))=0, \; \pi(Fix(a_{3}))=1,$$
$$\pi(Fix(a_{4}))=\lambda_{1},\ldots,\; \pi(Fix(a_{n+1}))=\lambda_{n-2}.$$

\s
\noindent
\begin{theorem}[\cite{GHL}]\label{teo6}
Within the above notations the following hold.
\begin{enumerate}
\item The generalized Fermat pairs
$(S,H)$ and $(C^{p}(\lambda_{1},\ldots,\lambda_{n-2}),H_0)$ are holomorphically equivalent. 

\item  ${\rm Aut}_{H_{0}}(C^{p}(\lambda_{1},\ldots,\lambda_{n-2}))/H_{0}$ is isomorphic to the subgroup of M\"obius transformations that preserves the finite set
$$\{\infty,0,1,\lambda_{1},\lambda_{2},\ldots,\lambda_{n-2}\}.$$
\item ${\rm Aut}_{H_{0}}(C^{p}(\lambda_{1},\ldots,\lambda_{n-2}))<PGL(n+1,{\mathbb C})$.
\end{enumerate}
\end{theorem}

\s
\noindent
\begin{remark}
Recall from Remark \ref{rema:unico} that ${\rm Aut}_{H_{0}}(C^{p}(\lambda_{1},\ldots,\lambda_{n-2}))={\rm Aut}(C^{p}(\lambda_{1},\ldots,\lambda_{n-2}))$.
\end{remark}

\s
\subsection{Principally polarized abelian varieties}\label{sec:poincare}
A {\it principally polarized abelian variety} of dimension $g \geq 1$ is a pair $A=(T,Q)$, where $T={\mathbb C}^{g}/L$ is a complex torus of dimension $g$ and $Q$ (called a {\it principal polarization} of $A$) is a positive-definite Hermitian product in ${\mathbb C}^{g}$ whose imaginary part ${\rm Im}(Q)$ has integral values over elements of the lattice $L$ and such that there is basis of $L$ for which ${\rm Im}(Q)$ can be represented by the matrix
$$\left( \begin{array}{cc}
0_{g} & I_{g}\\
-I_{g} & 0_{g}
\end{array}
\right)
$$

Two principally polarized abelian varieties $A_{1}$ and $A_{2}$ are called {\it isogenous} if there is a 
non-constant surjective morphism $h:A_{1} \to A_{2}$ (between the corresponding tori) with finite kernel; in this case $h$  
is called an {\it isogeny} (an isomorphism if the kernel is trivial)

A principally polarized abelian variety $A$ is called {\it decomposable} if it is isogenous to the product of abelian varieties of smaller dimensions (otherwise, it is said to be {\it simple}). It is called {\it completely decompossable} if it is the product of elliptic curves (varieties of dimension $1$).

\s
\noindent
\begin{theorem}[Poincar\'e's complete reducibility theorem \cite{Poincare}]
If $A$ is a principally polarized abelian variety, then there exist simple polarized abelian varieties $A_{1}, \ldots, A_{s}$ and positive integers $n_{1},\ldots,n_{s}$ such that $A$ is isogenous to the product $A_{1}^{n_{1}} \times \cdots A_{s}^{n_{s}}$. Moreover, the $A_{j}$ and $n_{j}$ are unique (up to isogeny) and permutation of the factors.
\end{theorem}

\s

We should restrict to certain small locus of principally polarized abelian varieties coming from closed Riemann surfaces.

\subsection{The Jacobian variety}\label{sec:jacobian}
Let $S$ be a closed Riemann surface of genus $g \geq 1$. Its first homology group $H_{1}(S,{\mathbb Z})$ is isomorphic, as a ${\mathbb Z}$-module, to ${\mathbb Z}^{2g}$ and the complex vector space $H^{1,0}(S)$ of its holomorphic $1$-forms is isomorphic to ${\mathbb C}^{g}$. There is a natural injective map
$$\iota:H_{1}(S,{\mathbb Z}) \hookrightarrow \left( H^{1,0}(S) \right)^{*} \quad \mbox{(the dual space of $H^{1,0}(S)$)}$$
$$\alpha \mapsto \int_{\alpha}.$$

The image $\iota(H_{1}(S,{\mathbb Z}))$ is a lattice in $\left( H^{1,0}(S) \right)^{*}$ and the quotient $g$-dimensional torus 
$$JS=\left( H^{1,0}(S) \right)^{*}/\iota(H_{1}(S,{\mathbb Z}))$$
is called the Jacobian variety of $S$.  The intersection product form in $H_{1}(S,{\mathbb Z})$ induces a principal polarization on $JS$.

If we fix a point $p_{0} \in S$, then there is a natural holomorphic embedding $$\rho_{p_{0}}:S \to JS$$
defined by $\rho(p)=\int_{\alpha}$, where $\alpha \subset S$ is an arc connecting $p_{0}$ with $p$.

If we choose a symplectic homology basis for $S$, say $\{\alpha_{1},\ldots,\alpha_{g},\beta_{1},\ldots,\beta_{g}\}$ (i.e. a basis for $H_{1}(S,{\mathbb Z})$ such that the intersection products $\alpha_{i} \cdot \alpha_{j}=\beta_{i} \cdot \beta_{j}=0$ and $\alpha_{i} \cdot \beta_{j}=\delta_{ij}$, where $\delta_{ij}$ is the Kronecker delta function), we may find a dual basis $\{\omega_{1},\ldots,\omega_{g}\}$ (i.e. a basis of $H^{1,0}(S)$ such that $\int_{\alpha_{i}} \omega_{j}=\delta_{ij}$).  We may consider the Riemann matrix 
$$Z=\left( \int_{\beta_{j}} \omega_{i} \right)_{g \times g} \in {\mathfrak H}_{g}.$$

If we now consider the Riemann period matrix $\Omega=(I \; Z)_{g \times 2g}$, then its $2g$ columns define a lattice in ${\mathbb C}^{g}$. The quotient torus ${\mathbb C}^{g}/\Omega$ is isomorphic to $JS$.
 
\subsection{Holomorphic maps and the Jacobian variety}
Let $f:S \to R$ be a non-constant holomorphic map between closed Riemann surfaces, both of genus at least $1$. There are natural induced maps 
$$H_{1}(f):H_{1}(S,{\mathbb Z}) \to H_{1}(R,{\mathbb Z})$$
$$H^{1,0}(f):H^{1,0}(R) \to H^{1,0}(S) \quad \mbox{(pull-back of forms)}$$
which together (we use the dual map of the last) permit to define a non-constant surjective morphism
$$J(f):JS \to JR$$
such that, if $p_{0} \in S$, then $$J(f) \circ \rho_{p_{0}} = \rho_{f(p_{0})} \circ f.$$

The connected component of the kernel of $J(f)$ containing $0$ is called the Prym sub-variety of $JS$ with respect to $f$.

\subsection{Kani-Rosen's decomposition theorem}
As a consequence of Poincar\'e complete reducibility theorem, the jacobian variety of a closed Riemann surface can be decomposed, up to isogeny, into a product of simple sub-varieties. To obtain such a decomposition is not an easy job, but there are general results which permit to work in this direction.

The following result, due to Kani and Rosen \cite{K-R}, provides sufficient conditions for the Jacobian variety of a closed Riemann surface to decompose into the product of the Jacobian varieties of suitable quotient Riemann surfaces. If $K<{\rm Aut}(S)$, we denote by $g_{K}$ the genus of the quotient orbifold $S/K$ and by $S_{K}$ the underlying Riemann surface structure of the orbifold $S/K$.

\s
\noindent
\begin{theorem}[Kani-Rosen's decomposition theorem \cite{K-R}]
Let $S$ be a closed Riemann surface of genus $g \geq 1$ and let $H_{1},\ldots,H_{s}<{\rm Aut}(S)$ such that:
\begin{enumerate}
\item $H_{i} H_{j}=H_{j} H_{i}$, for all $i,j =1,\ldots,s$;
\item there are integers $n_{1},\ldots, n_{s}$ satisfying that
\begin{enumerate}
\item $\sum_{i,j=1}^{s} n_{i}n_{j} g_{H_{i}H_{j}}=0$, and
\item for every $i=1,\ldots,s$, it also holds that $\sum_{j=1}^{s} n_{j} g_{H_{i}H_{j}}=0$.
\end{enumerate}
\end{enumerate}

Then 
$$\prod_{n_{i}>0} \left( JS_{H_{i}} \right)^{n_{i}} \cong_{isog.}  \prod_{n_{j}<0} \left( JS_{H_{j}} \right)^{-n_{j}}.$$
\end{theorem}

\s

If in the above theorem we set $H_{s}=\{1\}$, $n_{1}=\cdots=n_{s-1}=-1$ and $n_{s}=1$, then one has the following consequence that we will use for our family of generalized Fermat curves.

\s
\noindent
\begin{corollary}[\cite{K-R}]\label{coroKR}
Let $S$ be a closed Riemann surface of genus $g \geq 1$ and let $H_{1},\ldots,H_{t}<{\rm Aut}(S)$ such that:
\begin{enumerate}
\item $H_{i} H_{j}=H_{j} H_{i}$, for all $i,j =1,\ldots,t$;
\item $g_{H_{i}H_{j}}=0$, for $1 \leq i < j \leq t$
\item $g=\sum_{j=1}^{t} g_{H_{j}}$.
\end{enumerate}

Then 
$$JS \cong_{isog.}  \prod_{j=1}^{t} JS_{H_{j}}.$$
\end{corollary}

\s

\section{A simple (well known) application to hyperelliptic Riemann surfaces}\label{sec:hiper} 
In this section we describe (some of them well known) results on the decomposition of the jacobian variety of hyperelliptic Riemann surface with automorphisms of order two different from the hyperelliptic involution, as an application of Corollary \ref{coroKR}, and we provide some examples whose Jacobian variety decomposes as product of elliptic curves (in genera two and four).

\subsection{A decomposition for hyperelliptic Riemann surfaces with an extra involution}
\s
\noindent
\begin{corollary}\label{hiper2}
Let $S$ be a hyperelliptic Riemann surface of genus $g \geq 2$, with hyperelliptic involution $\iota$, admitting an extra conformal involution $\tau$. 

\begin{enumerate}
\item There are values $\lambda_{1},\ldots,\lambda_{g} \in {\mathbb C}-\{0,1\}$, with  $\lambda_{i} \neq \lambda_{j}$ for $i \neq j$, such that $S$ can be described by the hyperelliptic curve 
\begin{equation}\label{ecuacion}
C: y^{2}=x(x-1)\left(x+\frac{(1-\mu_{1})^{2}}{4\mu_{1}}\right) \prod_{j=2}^{g} 
\left(x^{2}-\frac{(1-\mu_{1})(\mu_{j}^{2}+\mu_{1})}{\mu_{j}^{2}-\mu_{1}^{2}} x  + \frac{(1-\mu_{1})^{2}(\mu_{j}^{2}-1)}{4(\mu_{j}^{2}-\mu_{1}^{2})}  \right),
\end{equation}
where 
$$\mu_{1}=\sqrt{\frac{\lambda_{2\gamma}}{\lambda_{2\gamma-1}} \left(\frac{\lambda_{2\gamma-1}-1}{\lambda_{2\gamma}-1}\right)}, \quad 
\mu_{2}=\sqrt{\frac{\lambda_{2\gamma-1}-1}{\lambda_{2\gamma}-1}},$$
$$\mu_{j}=\sqrt{\left(\frac{\lambda_{2\gamma-1}-1}{\lambda_{2\gamma}-1}\right) \left( \frac{\lambda_{j-2}-\lambda_{2\gamma}}{\lambda_{j-2}-\lambda_{2\gamma-1}}\right)}, \quad j=3,...,g.$$

Observe that, in this model of $C$ one has that $\iota(x,y)=(x,-y)$ and 

$$\tau(x,y)=$$
$${\tiny \left( \frac{(\mu_{1}-1)^{2}(1-x)}{(4\mu_{1} x + (\mu_{1}-1)^{2}},
y \left( \frac{2\sqrt{\mu_{1}}(1-\mu_{1}^{2})^{2}}{(4 \mu_{1} x+(1-\mu_{1})^{2})^{g+1}}
\right) \prod_{j=2}^{g} \frac{(1-\mu_{1})}{\sqrt{\mu_{j}^{2}-\mu_{1}^{2}}}
\sqrt{\left((\mu_{j}^{2}-\mu_{1}^{2})(1-\mu_{1})^{2}+4\mu_{1}^{2} (1-\mu_{j}^{2})+4\mu_{1}(1-\mu_{1})(\mu_{1}+\mu_{j}^{2})  \right)}
 \right).}
 $$

\item If $g=2\gamma$, then, 
in this model, the underlying hyperelliptic Riemann surfaces of the orbifolds $S/\langle \tau \rangle$  and $S/\langle \tau\iota \rangle$
are respectively described by the curves
$$C_{1}: y^{2}=x(x-1) (x-\lambda_{g})\prod_{j=1}^{g-2} (x-\lambda_{j}),$$
$$C_{2}: y^{2}=x(x-1) (x-\lambda_{g-1}) \prod_{j=1}^{g-2} (x-\lambda_{j}).$$

\item If $g=2\gamma+1$, then, in this model,  
the underlying hyperelliptic Riemann surfaces of the orbifolds $S/\langle \tau \rangle$  and $S/\langle \tau\iota \rangle$
are respectively described by the curves
$$C_{1}: y^{2}=x(x-1) \prod_{j=1}^{g-2} (x-\lambda_{j}),$$
$$C_{2}: y^{2}=x(x-1) \prod_{j=1}^{g} (x-\lambda_{j}).$$

\item $JS \cong_{isog.} JC_{1} \times JC_{2}$.

\item $S$ is the fiber product of $C_{1}$ and $C_{2}$, where in each case we use the natural projection $(x,y) \mapsto x$.

\end{enumerate}

\end{corollary}
\begin{proof}
Let $S_{\tau}$ (respectively, $S_{\iota\tau}$) the underlying Riemann surface structure of the orbifold $S/\langle \tau \rangle$ (respectively, $S/\langle \iota\tau \rangle$). Take $H_{1}=\langle \tau \rangle$ and $H_{2}=\langle \iota\tau \rangle$. Since $H_{1}H_{2}=\langle \iota, \tau \rangle \cong {\mathbb Z}_{2}^{2}$, $S/H_{1}H_{2}$ has genus zero and $S/H_{i}$ has genus $\gamma$. It follows from Corollary \ref{coroKR} that $JS \cong_{isog.} JS_{\tau} \times JS_{\iota\tau}$; providing part (4) of the theorem. Part (5) is almost clear.

\s

We only provide the proof for part (1) and (2) for the case that $g=2\gamma$ as for the other the arguments are similar.
Let us now construct algebraic curve representations for $S$, $S_{\tau}$ and $S_{\iota\tau}$ as desired.

It is well known that the hyperelliptic Riemann surface $S$ can be described by a hyperelliptic curve of the form
$$C_{0}: y^{2}=(x^{2}-1)\prod_{j=1}^{2\gamma}(x^{2}-\mu_{j}^{2}),$$
where $\mu_{1},\ldots, \mu_{2\gamma} \in {\mathbb C}-\{0, \pm 1\}$ and, for $i \neq j$, $\mu^{2}_{i} \neq \mu^{2}_{j}$. In this algebraic model, 
$$\iota(x,y)=(x,-y), \quad \tau(x,y)=(-x,y),$$
the regular branched two-cover defined by the hyperelliptic involution is given by
$$\pi(x,y)=x$$
and the induced involution, under $\pi$, by $\tau$ is
$$\widetilde{\tau}(x)=-x.$$

If we consider the M\"obius transformation 
$$T(x)=\left(\frac{1-\mu_{1}}{2}\right) \left(\frac{x+1}{x-\mu_{1}}\right),$$
then, since
$$T(1)=1,\; T(-1)=0, \; T(\mu_{1})=\infty, \; 
T(-\mu_{1})=-\frac{(1-\mu_{1})^{2}}{4\mu_{1}},$$
and for $j=2,\ldots,2\gamma$, 
$$T(\mu_{j})=\left( \frac{1-\mu_{1}}{2} \right) \left( \frac{\mu_{j}+1}{\mu_{j}-\mu_{1}}  \right),\;
T(-\mu_{j})=\left( \frac{1-\mu_{1}}{2} \right) \left( \frac{\mu_{j}-1}{\mu_{j}+\mu_{1}}  \right),$$
$S$ can be described by the hyperelliptic curve \eqref{ecuacion}.

A regular branched two-cover $P:C_{0} \to \widehat{\mathbb C}$, induced by $\widetilde{\tau}$, is given by
$$P(x)=\left( \frac{1-\mu_{2}^{2}}{1-\mu_{1}^{2}}\right) \left( \frac{x^{2}-\mu_{1}^{2}}{x^{2}-\mu_{2}^{2}}\right).$$ 

We have chosen $P$ so that $P(\pm 1)=1$, $P(\pm \mu_{1})=0$ and $P(\pm \mu_{2})=\infty$.

Let us set, for $j=3,...,2\gamma$,  
$$\lambda_{j-2}=P(\pm \mu_{j})=\left( \frac{1-\mu_{2}^{2}}{1-\mu_{1}^{2}}\right) \left( \frac{\mu_{j}^{2}-\mu_{1}^{2}}{\mu_{j}^{2}-\mu_{2}^{2}}\right),$$ 
and
$$\lambda_{2\gamma-1}=P(\infty)=\left( \frac{1-\mu_{2}^{2}}{1-\mu_{1}^{2}}\right), \quad
\lambda_{2\gamma}=P(0)=\frac{\mu_{1}^{2}}{\mu_{2}^{2}} \left( \frac{1-\mu_{2}^{2}}{1-\mu_{1}^{2}}\right).$$

The underlying hyperelliptic Riemann surface of the orbifold $S/\langle \tau \rangle$ is described by the hyperelliptic curve
$$C_{1}: y^{2}=x(x-1) \left( \prod_{j=1}^{2\gamma-2} (x-\lambda_{j}) \right) (x-\lambda_{2\gamma})$$ 
and the underlying hyperelliptic Riemann surface of the orbifold $S/\langle \tau\iota \rangle$ is described by the hyperelliptic curve
$$C_{2}: y^{2}=x(x-1) \left( \prod_{j=1}^{2\gamma-2} (x-\lambda_{j}) \right) (x-\lambda_{2\gamma-1}).$$ 

The equalities $P(0)=\lambda_{2\gamma}$ and $P(\infty)=\lambda_{2\gamma-1}$ imply that
$$\mu_{1}^{2}=\frac{\lambda_{2\gamma}}{\lambda_{2\gamma-1}} \left(\frac{\lambda_{2\gamma-1}-1}{\lambda_{2\gamma}-1}\right), \quad 
\mu_{2}^{2}=\frac{\lambda_{2\gamma-1}-1}{\lambda_{2\gamma}-1}.$$

Also, the equality $P(\pm \mu_{j})=\lambda_{j-2}$, for $j=3,...,2\gamma$, ensures that
$$\mu_{j}^{2}=\left(\frac{\lambda_{2\gamma-1}-1}{\lambda_{2\gamma}-1}\right) \left( \frac{\lambda_{j-2}-\lambda_{2\gamma}}{\lambda_{j-2}-\lambda_{2\gamma-1}}\right).$$

The involution $\tau$ can be directly computed.

\end{proof}

\s
\subsection{Genus two hyperelliptic Riemann surfaces}\label{genustwo}
Corollary \ref{hiper2}, in particular, provides the very well known fact that a closed Riemann surface of genus two admitting a conformal involution, different from the hyperelliptic one, has its Jacobian varity isogenous to the product of two elliptic curves (see for instance \cite{Riera-Rodriguez}).

\s
\subsection{Genus four hyperelliptic Riemann surfaces}
Let us consider a hyperelliptic Riemann surface $S$ of genus four admitting a conformal involution $\tau$ with exactly two fixed points. If $S_{1}$ and $S_{2}$ are the hyperelliptic Riemann surface structures of the orbifolds $S/\langle \tau \rangle$ and $S/\langle \tau\iota \rangle$, then Corollary \ref{hiper2} asserts that 
$JS \cong_{isog.} JS_{1} \times JS_{2}$. Now, if the Riemann surface $S_{j}$, for $j=1,2$, admits a conformal involution with two fixed points, then again $JS_{j} \cong_{isog.} E_{1,j} \times E_{2,j}$, where $E_{i,j}$ is elliptic curve and it will follow that  $JS \cong_{isog.} E_{1,1} \times E_{1,2} \times E_{2,1} \times E_{2,2}$. Next result provides explicitly such kind of examples.

\s
\noindent
\begin{corollary}\label{genusfour}
Let $\lambda_{1,1}, \lambda_{1,2} \in {\mathbb C}-\{0,1\}$ so that $\lambda_{1,1} \neq \lambda_{1,2}$ and $\lambda_{1,1}\lambda_{1,2} \neq 1$. Set
$$\lambda_{2,1}=\frac{1}{2} \left(-4\lambda_{1,2}+\frac{4+2\lambda_{1,1}-13\lambda_{1,2}+8\lambda_{1,2}^{2}-\lambda_{1,2}^{3}}{1-4\lambda_{1,2}+2\lambda_{1,1}\lambda_{1,2}+\lambda_{1,2}^{2}} \right)$$
$$\lambda_{2,2}=\frac{2\lambda_{1,1}+\lambda_{1,2}-4\lambda_{1,1}\lambda_{1,2}+\lambda_{1,1}^{2}\lambda_{1,2}}{1-4\lambda_{1,1}+2\lambda_{1,1}\lambda_{1,2}+\lambda_{1,1}^{2}}.$$

We also assume that $\lambda_{2,1}, \lambda_{2,2} \in {\mathbb C}-\{0,1\}$, $\lambda_{2,1} \neq \lambda_{1,1}$ and $\lambda_{2,2} \neq \lambda_{1,2}$ (this provides an open dense subspace in the $(\lambda_{1,1},\lambda_{1,2})$-space). Set

$$\lambda_{1}=-\frac{(1-\mu_{1,1})^{2}}{4\mu_{1,1}}; \;
\lambda_{2}=\left(\frac{1-\mu_{1,1}}{2}\right)\left(\frac{\mu_{2,1}+1}{\mu_{2,1}-\mu_{1,1}}\right);$$
$$\lambda_{2+j}=\left(\frac{1-\mu_{1,j}}{2}\right)\left(\frac{\mu_{2,j}-1}{\mu_{2,j}+\mu_{1,j}}\right), \; j=1,2,$$
where
$$\mu_{1,j}=\sqrt{\frac{\lambda_{2,j}}{\lambda_{1,j}} \left(\frac{\lambda_{1,j}-1}{\lambda_{2,j}-1}\right)}; \quad 
\mu_{2,j}=\sqrt{\frac{\lambda_{1,j}-1}{\lambda_{2,j}-1}}.$$

Also set
$$\mu_{1}=\sqrt{\frac{\lambda_{4}}{\lambda_{3}} \left(\frac{\lambda_{3}-1}{\lambda_{4}-1}\right)}; \quad 
\mu_{2}=\sqrt{\frac{\lambda_{3}-1}{\lambda_{4}-1}};$$
$$\mu_{3}=\sqrt{\left(\frac{\lambda_{3}-1}{\lambda_{4}-1}\right) \left(\frac{\lambda_{1}-\lambda_{4}}{\lambda_{1}-\lambda_{3}}\right)}; \quad \mu_{4}=\sqrt{\left(\frac{\lambda_{3}-1}{\lambda_{4}-1}\right) \left(\frac{\lambda_{2}-\lambda_{4}}{\lambda_{2}-\lambda_{3}}\right)}.$$

Let $S$ be the genus four hyperelliptic curve 
\begin{equation}
C: y^{2}=x(x-1)\left(x+\frac{(1-\mu_{1})^{2}}{4\mu_{1}}\right) \prod_{j=2}^{4} 
\left(x^{2}-\frac{(1-\mu_{1})(\mu_{j}^{2}+\mu_{1})}{\mu_{j}^{2}-\mu_{1}^{2}} x  + \frac{(1-\mu_{1})^{2}(\mu_{j}^{2}-1)}{4(\mu_{j}^{2}-\mu_{1}^{2})}  \right).
\end{equation}

Then $$JC \cong_{isog.} C_{11} \times C_{12} \times C_{21} \times C_{22},$$
where
$$C_{11}: y^{2}=x(x-1)(x-\lambda_{1,2}); \quad 
C_{12}: y^{2}=x(x-1)(x-\lambda_{1,1});$$
$$C_{21}: y^{2}=x(x-1)(x-\lambda_{2,2}); \quad 
C_{22}: y^{2}=x(x-1)(x-\lambda_{2,1}).$$

\end{corollary}
\begin{proof}
Let us start observing that in order for $S_{j}$ to admit a conformal involution with exactly two fixed points it must have an equation as follows (by Corollary \ref{hiper2})

$$C_{j}: y^{2}=x(x-1)( x-\rho_{1,j})(x-\rho_{2,j})(x-\rho_{3,j}),$$
where 
$$\rho_{1,j}=-\frac{(1-\mu_{1,j})^{2}}{4\mu_{1,j}}; \quad 
\rho_{2,j}=\left(\frac{1-\mu_{1,j}}{2}\right)\left(\frac{\mu_{2,j}+1}{\mu_{2,j}-\mu_{1,j}}\right); \quad
\rho_{3,j}=\left(\frac{1-\mu_{1,j}}{2}\right)\left(\frac{\mu_{2,j}-1}{\mu_{2,j}+\mu_{1,j}}\right);$$
$$\mu_{1,j}=\sqrt{\frac{\lambda_{2,j}}{\lambda_{1,j}}\left(\frac{\lambda_{1,j}-1}{\lambda_{2,j}-1}\right)}; \quad
\mu_{2,j}=\sqrt{\frac{\lambda_{1,j}-1}{\lambda_{2,j}-1}};$$
and
$\lambda_{1,j},\lambda_{2,j} \in {\mathbb C}-\{0,1\}$ with $\lambda_{1,j} \neq \lambda_{2,j}$.

At this point, we are working with $4$ parameters; $\lambda_{1,1}$, $\lambda_{2,1}$, $\lambda_{1,2}$ and $\lambda_{2,2}$.

Let $\iota_{j}$ be the hyperelliptic involution of $S_{j}$ (which is induced by $\iota$). The fixed points of $\iota_{j}$ are the images in $S_{j}$ of the fixed points of $\iota$ (this produces five of them) and the other is the  image of the fixed points of $\iota\tau$ for $j=1$ and the image of the fixed points of $\tau$ for $j=2$. 

The quotient $S_{j}/\langle \iota_{j} \rangle$ is the Riemann sphere whose cone points are $\infty$, $0$, $1$, $\rho_{1,j}$, $\rho_{2,j}$ (these are the projections of those fixed points of $\iota_{j}$ which are comming from the fixed points of $\iota$) and a point $\rho_{3,j}$, which is the image of the fixed point of $\iota_{j}$ which is comming from the fixed points of $\iota\tau$ for $j=1$ and the image of the fixed points of $\tau$ for $j=2$.

Since $\{\infty,0,1,\rho_{1,1},\rho_{2,1},\rho_{3,1},\rho_{1,2},\rho_{2,2},\rho_{3,2}\}$ are also the cone points of $S/\langle \iota, \tau \rangle$, we may see that 
$\{\rho_{1,1},\rho_{2,1}\} =\{\rho_{1,2},\rho_{2,2}\}$, $\rho_{3,1} \neq \rho_{3,2}$.

Let us assume that $\rho_{1,1}=\rho_{2,2}$ and $\rho_{2,1}=\rho_{1,2}$. This is equivalent to have
$$\lambda_{2,1}=\frac{1}{2} \left(-4\lambda_{1,2}+\frac{4+2\lambda_{1,1}-13\lambda_{1,2}+8\lambda_{1,2}^{2}-\lambda_{1,2}^{3}}{1-4\lambda_{1,2}+2\lambda_{1,1}\lambda_{1,2}+\lambda_{1,2}^{2}} \right)$$
$$\lambda_{2,2}=\frac{2\lambda_{1,1}+\lambda_{1,2}-4\lambda_{1,1}\lambda_{1,2}+\lambda_{1,1}^{2}\lambda_{1,2}}{1-4\lambda_{1,1}+2\lambda_{1,1}\lambda_{1,2}+\lambda_{1,1}^{2}}.$$

The above provides two equations for the four parameters $\lambda_{1,1}$, $\lambda_{2,1}$, $\lambda_{1,2}$ and $\lambda_{2,2}$.

The condition $\rho_{3,1} \neq \rho_{3,2}$ is now equivalent to have 
$$\lambda_{1,1} \neq \lambda_{1,2}, \quad \lambda_{1,1}\lambda_{1,2} \neq 1.$$

Now, Corollary \ref{hiper2} tells us how to write the equation for $S$ in terms of the above values, where in this case, $\lambda_{1}=\rho_{1,1}$, $\lambda_{2}=\rho_{2,1}$, $\lambda_{3}=\rho_{3,1}$ and $\lambda_{4}=\rho_{3,2}$.

\end{proof}

\s
\noindent
\begin{remark}
In \cite{Paulhus1, Paulhus2, Paulhus3} it is shown a particular hyperelliptic Riemann surface of genus four whose Jacobian variety is isogenous to a product $E_{1}^{2} \times E_{2}^{2}$, where $E_{1}$ and $E_{2}$ are elliptic curves. Such a curve is given by
$$P: y^{2}=x(x^{4}-1)(x^{4}+2\sqrt{-3} \; x^{2}+1),$$
but it does not have an automorphism of order two differently form the hyperelliptic involution; so it cannot be produced with our method. Anyway, if we consider, for instance, $\lambda_{1,1}=-1$ and $\lambda_{1,2}=1/5$, then 
$j(\lambda_{1,1})=j(\lambda_{2,1})=27/4$ and $j(\lambda_{1,2})=j(\lambda_{2,2})=9261/400$, where $j$ is the Klein-elliptic function
$$j(\lambda)=\frac{(1-\lambda+\lambda^{2})^{3}}{\lambda^{2}(1-\lambda)^{2}},$$ and the Jacobian variety of the hyperelliptic Riemann surface $C$ constructed in Corollary \ref{genusfour} will satisfy $JC \cong_{isog} C_{1,1}^{2} \times C_{1,2}^{2}$. Similar kind of examples can be produced from Corollary \ref{genusfour} by considering appropriated parameters $\lambda_{1,1}$ and $\lambda_{1,2}$. For instance, if we take $\lambda_{1,1}=4+\sqrt{11}$ and $\lambda_{1,2}=1-\lambda_{1,1}=-3-\sqrt{11}$,
then
$\lambda_{2,1}=(3-\sqrt{11})/2$ and $\lambda_{2,2}=(1+\sqrt{11})/2$; so
$j(\lambda_{1,1})=j(\lambda_{1,2})=j(\lambda_{2,1})=(1489+294 \sqrt{11})=50$ and $j(\lambda_{2,2})=343/50$, and $JC$ is isogenous to $C_{11}^{3} \times C_{22}$.

\end{remark}

\s

\section{Jacobian variety of generalized Fermat curves of type $(p,n)$, with $p$ a prime integer}
In this section we describe the main result of the paper, that is, an isogenous decomposition of the Jacobian variety of a generalized Fermat curve of type $(p,n)$, where $p$ is a prime integer. In the next two sections we make this explicit for $p \in \{2,3\}$. 

\s
\subsection{A counting formula}

\s
\noindent
\begin{lemma}\label{formulita1}
Let $q \geq 2$ and $r \geq 2$ be integers and let 
$\psi_{q}(r)$ be the number of different tuples $(\alpha_{2},\ldots,\alpha_{r})$ so that 
$\alpha_{j} \in \{1,2,\ldots,q-1\}$, and $\alpha_{2}+\cdots+\alpha_{r} \equiv -1 \mod(q)$. 
Then 
$$
\psi_{q}(r)=(-1)^{r+1}\left( \frac{(1-q)^{r-1}-1}{q}\right).
$$
\end{lemma}
\begin{proof}
Let us consider a tuple $(\alpha_{2},\ldots,\alpha_{r-1},\alpha_{r})$, where $\alpha_{j} \in \{1,\ldots,q-1\}$ and
$\alpha_{2}+\cdots+\alpha_{r} \equiv -1 \mod(q)$. Since $\alpha_{r}$ is not congruent to $0$ mod $q$, we must have that 
$\alpha_{2}+\cdots+\alpha_{r-1}$ cannot be congruent to $-1$ mod $q$. But this last sum can be congruent to any value inside $\{0,1,\ldots,q-2\}$. We also note that $\alpha_{r}$ gets uniquely determined by $\alpha_{1},\ldots, \alpha_{r-1}$. In this way, 
$\psi_{q}(r)=(q-1)^{r-2}-\psi_{q}(r-1)$. This recurrence asserts that
$$\psi_{q}(r)=\sum_{k=2}^{r} (-1)^{k} (q-1)^{r-k}=(-1)^{r}\sum_{k=2}^{r} (1-q)^{r-k}=(-1)^{r}\sum_{k=0}^{r-2} (1-q)^{k}=(-1)^{r+1}\left( \frac{(1-q)^{r-1}-1}{q}\right).$$

\end{proof}

\s

For instance, $\psi_{2}(2s)=1$, $\psi_{2}(2s-1)=0$, and for $q \geq 3$ we have $\psi_{q}(2)=1$, $\psi_{q}(3)=q-2$, $\psi_{q}(4)=q^{2}-3q+3$ and $\psi_{q}(5)=q^{3}-4q^{2}+6q-4$. 

\s

We will need the following equality, which permits to write the genus of a generalized Fermat curve as the sum of the genus of cyclic gonal curves (we will use this for the prime case). 

\s
\noindent
\begin{lemma}\label{conteo}
Let $n,q \geq 2$ be integers with $n+1 \geq r_{q}$, where $r_{2}=4$ and $r_{q}=3$ for $q \geq 3$. Then
$$1+ \frac{\phi(q,n)}{2}=\sum_{r=r_{q}}^{n+1}  \binom{n+1}{r}  \frac{(r-2)(q-1)}{2} \psi_{q}(r).$$
\end{lemma}
\begin{proof} By Lemma \ref{formulita1}, the equality we want to obtain is 
$$
1+\frac{q^{n-1}((n-1)q-n-1)}{2}=\frac{(q-1)}{2q}\sum_{r=r_{q}}^{n+1}  \binom{n+1}{r}  (r-2)((q-1)^{r-1}-(-1)^{r-1}).
$$

Let us consider the function 
$$f(x)=\frac{(q-1)(1+x)^{n+1}}{2qx^{2}}=\frac{(q-1)}{2q}\sum_{r=0}^{n+1} \binom{n+1}{r} x^{r-2}.$$

Then taking the derivative with respect to $x$, the above provides the following equality
$$\frac{(q-1)(1+x)^{n}((n-1)x-2)}{2qx^{3}}=\frac{(q-1)}{2q}\sum_{r=0}^{n+1} \binom{n+1}{r} (r-2)x^{r-3}.$$

Evaluating the last one at $x=q-1$ provides the equality
$$\frac{q^{n-1}((n-1)q-n-1)}{2}=\frac{(q-1)}{2q}\sum_{r=0}^{n+1} \binom{n+1}{r}(r-2)(q-1)^{r-1}$$
and evaluating at $x=-1$ provides the equality
$$0=\frac{(q-1)}{2q}\sum_{r=0}^{n+1} \binom{n+1}{r}(r-2)(-1)^{r-1}.$$

By subtracting both equalities one obtains 
$$\frac{q^{n-1}((n-1)q-n-1)}{2}=\frac{(q-1)}{2q}\sum_{r=0}^{n+1}  \binom{n+1}{r}  (r-2)((q-1)^{r-1}-(-1)^{r-1}).
$$
from which the desired equality follows.
\end{proof}

\s
\noindent
\begin{remark}
A way to see $\psi_{p}(r)$, where $p \geq 2$ is  prime,  in terms of the generalized Fermat group is the following. Let us consider a generalized Fermat curve $S$ of type $(p,n)$ and its corresponding generalized Fermat group $H_{0}=\langle a_{1},\ldots,a_{n}\rangle \cong {\mathbb Z}_{p}^{n}$. Let us fix $r$ of the cone points $\infty$, $0$, $1$, $\lambda_{1}, \ldots, \lambda_{n-2}$. Let $R$ be a closed Riemann surface admitting a group $K \cong {\mathbb Z}_{p}$ as a group of conformal automorphisms so that $R/K$ is the Riemann sphere whose cone points are exactly the chosen $r$ points. Let $P:R \to \widehat{\mathbb C}$ be a regular branched cover with $K$ as its deck group and whose branch values are these $r$ points. If we lift the $n+1-r$ other points via $P$ to $R$, then we obtain a Riemann orbifold ${\mathcal O}_{R}$, of signature $((r-2)(p-1)/2;p,\stackrel{p(n+1-r)}{\ldots},p)$, admitting $K$ as a group of conformal automorphisms. It follows the existence of a subgroup $H_{K} \cong {\mathbb Z}_{p}^{n-1}$ of $H_{0}$ so that ${\mathcal O}_{R}=S/H_{K}$. In this way, $\psi_{p}(r)$ equals to the number of subgroups $H \cong {\mathbb Z}_{p}^{n-1}$ of $H_{0}$ with $S/H$ of signature $((r-2)(p-1)/2;p,\stackrel{p(n+1-r)}{\ldots},p)$.
\end{remark}

\s
\subsection{Some cyclic $p$-gonal covers}
Let $p \geq 2$ be a prime integer and set $r_{2}=4$ and, for $p \geq 3$, set $r_{p}=3$. 

Let us fix an 
integer $r \geq r_{p}$ (we assume $r$ even if $p=2$) and we also fix a collection of $r$  pairwise different complex numbers $\mu_{1},\ldots,\mu_{r} \in {\mathbb C}$. 

Let $S$ be closed Riemann surface $S$ admitting a group $K \cong {\mathbb Z}_{p}$ as a group of conformal automorphism so that $S/K$ has signature $(0;p,\stackrel{r}{\ldots},p)$ and whose cone points are $\mu_{1}$,..., $\mu_{r}$. By the Riemann-Hurwitz formula, the genus of $S$ is
$\gamma=(r-2)(p-1)/2$. Also, the surface
$S$ can be described by a cyclic $p$-gonal curve of the form
$$C(\alpha_{1},\ldots,\alpha_{r}): y^{p}=\prod_{j=1}^{r}(x-\mu_{j})^{\alpha_{j}},$$
where $\alpha_{j} \in \{1,2,\ldots,p-1\}$ and $\alpha_{1}+\cdots+\alpha_{r} \equiv 0 \mod(p)$.  In this curve model, a generator of $K$ is given by $\tau(x,y)=(x,e^{2\pi i/p}y)$. 

If $\beta \in \{1,2,\ldots,p-1\}$, then the curves $C(\alpha_{1},\ldots,\alpha_{r})$ and $C(\beta\alpha_{1},\ldots,\beta\alpha_{r})$ define isomorphic Riemann surfaces. It follows that we may assume $\alpha_{1}=1$. Lemma \ref{formulita1} tell us that the number of such (normalized) different cylic $p$-gonal curves is $\psi_{p}(r)$.

We should note that two curves $C(1,\alpha_{2},\ldots,\alpha_{r})$ and $C(1,\beta_{2},\ldots,\beta_{r})$ may be isomorphic even if the ordered tuples $(\alpha_{2},\ldots,\alpha_{r})$ and $(\beta_{2},\ldots,\beta_{r})$ are different; so $\psi_{p}(r)$ is not in general equal to the number of isomorphic classes of cyclic $p$-gonal curves.
\s

\subsection{Decomposition structure of the Jacobian variety of a generalized Fermat curves of type $(p,n)$, with $p \geq 2$ a prime integer}
Let us consider a generalized Fermat curve of type $(p,n)$, with $n+1\geq r_{p}$, say 
$$
S:=C^{p}(\lambda_{1},\ldots,\lambda_{n-2})=\left\{ \begin{array}{ccc}
x_{1}^{p}+x_{2}^{p}+x_{3}^{p}&=&0\\
\lambda_{1}x_{1}^{p}+x_{2}^{p}+x_{4}^{p}&=&0\\
\lambda_{2}x_{1}^{p}+x_{2}^{p}+x_{5}^{p}&=&0\\
\vdots & \vdots& \vdots\\
\lambda_{n-2}x_{1}^{p}+x_{2}^{p}+x_{n+1}^{p}&=&0
\end{array}
\right\} \subset {\mathbb P}^{n},
$$
where $\lambda_{1},\ldots,\lambda_{n-2} \in {\mathbb C}-\{0,1\}$ and, for $i \neq j$, $\lambda_{i} \neq \lambda_{j}$. The Riemann surface $S$ has genus $g=1+\phi(p,n)/2 \geq 1$ and it admits, as a group of conformal automorphisms, the generalized Fermat group $H_{0} \cong {\mathbb Z}_{p}^{n}$; generated by the transformations
$$a_{j}([x_{1}:\cdots:x_{n+1}])=[x_{1}:\cdots:x_{j-1}: \omega_{p} x_{j}:x_{j+1}:\cdots:x_{n+1}], \; j=1,\ldots,n,$$ 
where $\omega_{p}=e^{2 \pi i/p}$. We set $a_{n+1}=a_{1}^{-1}\cdots a_{n}^{-1}$, that is,
$$a_{n+1}([x_{1}:\cdots:x_{n+1}])=[x_{1}:\cdots: x_{n}: \omega_{p} x_{n+1}].$$

Let us make a choice of $r$ points inside $\{\infty,0,1,\lambda_{1},\ldots,\lambda_{n-2}\}$, where $n+1 \geq r \geq r_{p}$. We assume $r$ even if $p=2$. If $R$ is a closed Riemann surface admitting a group $K \cong {\mathbb Z}_{p}$ as group of conformal automorphism so that $R/K$ has signature $(0;p,\stackrel{r}{\ldots},p)$ and cone points are the $r$ chosen points, then there must be a subgroup $H_{R} \cong {\mathbb Z}_{p}^{n-1}$ of $H_{0}$ so that the quotient orbifold $S/H_{R}$ has underlying Riemann surface isomorphic to $R$ and so that the quotient orbifold under the action of $H/H_{R} \cong {\mathbb Z}_{p}$ over such a Riemann surface structure has signature $(0;p,\stackrel{r}{\ldots},p)$ and cone points the $r$ chosen points. Such a group $H_{R}$ must contain the elements $a_{j}$ corresponding to the complement of the $r$ chosen points but cannot contains the ones corresponding to these chosen points.

For any two different subgroups $H_{1}, H_{2}$ of $H_{0}$, with $H_{1} \cong {\mathbb Z}_{p}^{n-1} \cong H_{2}$ and with the above quotient property (maybe for different values of $r$), we clearly must have that $H_{1}H_{2}=H_{0}$.

All the above and Lemma \ref{conteo} permits us to apply Corollary \ref{coroKR} in order to obtain the following.

\s
\noindent
\begin{theorem}\label{(p,n)}
Let $(S,H)$ be a generalized Fermat pair of type $(p,n)$, where $p$ is a prime integer. Then 
$$JS \cong_{isog.} \prod_{H_r} JS_{H_r},$$
where $H_r$ runs over all subgroups of $H_{0}$ which are isomorphic to ${\mathbb Z}_{p}^{n-1}$ and such that $S/H_r$ has genus at least one, and $S_{H_r}$ is the underlying Riemann surface of the orbifold $S/H_r$.

The cyclic $p$-gonal curves $S_{H_r}$ runs over all curves of the form
$$y^{p}=\prod_{j=1}^{r}(x-\mu_{j})^{\alpha_{j}},$$
where $\{\mu_{1},\ldots,\mu_{r}\} \subset \{\infty,0,1,\lambda_{1},\ldots,\lambda_{n-2}\}$, $\mu_{i} \neq \mu_{j}$ if $i \neq j$,  $\alpha_{j} \in \{1,2,\ldots,p-1\}$ satisfying the following.
\begin{enumerate}
\item[(i)] If every $\mu_{j} \neq \infty$, then $\alpha_{1}=1$, $\alpha_{2}+\cdots+\alpha_{r} \equiv p-1 \mod(p)$;
\item[(ii)] If some $\lambda_{a}=\infty$, then  $\alpha_{1}+\cdots+\alpha_{a-1}+\alpha_{a+1}+\alpha_{r} \equiv p-1 \mod(p)$.
\end{enumerate}  
\end{theorem}

\s
\noindent
\begin{example}[Classical Fermat curves]
Let us consider, as an example, the classical Fermat curve $S=F_{p}=\{x_{1}^{p}+x_{2}^{p}+x_{3}^{p}=0\} \subset {\mathbb P}^{2}$, where $p \geq 2$ is a prime, which has genus $g=(p-1)(p-2)/2$. If ${\mathcal U}_{p}=\{1,\ldots,p-2\}$, then it is well known that $JS \cong_{isog.} \prod_{\alpha\in{\mathcal U}_{p}} A_{\alpha}$, where $A_{\alpha}$ are abelian varieties of CM type (in the sense of Shimura and Taniyama) (see, for instance, \cite{Rohrlich}). In fact, $A_{\alpha}=JS_{\alpha}$, where $S_{\alpha}$ is the cyclic $p$-gonal curve $$y^{p}=x(x-1)^{\alpha}.$$

This fact can be seen from Theorem \ref{(p,n)}. The above cyclic $p$-gonal curves correspond to the surfaces $F_{p}/\langle a_{1}a_{2}^{l}\rangle$, where $l \in \{2,\ldots,p-1\}$.

In the particular case $p=5$, each $C_{j}$ is isomorphic to the unique Riemann surface of genus two (up to isomorphisms) admitting an automorphism of order five; this being $C: y^{2}=x^{5}-1$. So, $JF_{5} \cong_{isog.} JC^{3}$.

\end{example}

\s

\section{Jacobian variety of generalized Fermat curves of type $(2,n)$}
In this section we make explicit computations for the case $p=2$.

\subsection{Generalized Fermat curves of type $(2,4)$}
In this case, 
\begin{equation}
S=\left\{ \begin{array}{ccc}
x_{1}^{2}+x_{2}^{2}+x_{3}^{2}&=&0\\
\lambda_{1}x_{1}^{2}+x_{2}^{2}+x_{4}^{2}&=&0\\
\lambda_{2}x_{1}^{2}+x_{2}^{2}+x_{5}^{2}&=&0
\end{array}
\right\} \subset {\mathbb P}^{4},
\end{equation}
where $\lambda_{1}, \lambda_{2} \in {\mathbb C}-\{0,1\}$ and $\lambda_{1} \neq \lambda_{2}$, 
has genus $g=5$ and the generalized Fermat group $H_{0} \cong {\mathbb Z}_{2}^{4}$ is generated by the transformations 
$$a_{1}([x_{1}:x_{2}:x_{3}:x_{4}:x_{5}])=[-x_{1}:x_{2}:x_{3}:x_{4}:x_{5}],$$ 
$$a_{2}([x_{1}:x_{2}:x_{3}:x_{4}:x_{5}])=[x_{1}:-x_{2}:x_{3}:x_{4}:x_{5}],$$ 
$$a_{3}([x_{1}:x_{2}:x_{3}:x_{4}:x_{5}])=[x_{1}:x_{2}:-x_{3}:x_{4}:x_{5}],$$ 
$$a_{4}([x_{1}:x_{2}:x_{3}:x_{4}:x_{5}])=[x_{1}:x_{2}:x_{3}:-x_{4}:x_{5}].$$ 

The transformation $a_{5}=a_{1}a_{2}a_{3}a_{4}$ is 
$$a_{5}([x_{1}:x_{2}:x_{3}:x_{4}:x_{5}])=[x_{1}:x_{2}:x_{3}:x_{4}:-x_{5}].$$ 

The list of the different subgroups $H$ of $H_{0}$, isomorphic to ${\mathbb Z}_{2}^{3}$ with $S/H$ of genus at least one is the following:
$$H_{1}=\langle a_{1}, a_{2}a_{3}, a_{2}a_{4} \rangle, \; 
H_{2}=\langle a_{2}, a_{1}a_{3}, a_{1}a_{4} \rangle$$
$$H_{3}=\langle a_{3}, a_{1}a_{2}, a_{2}a_{4} \rangle, \; 
H_{4}=\langle a_{4}, a_{2}a_{3}, a_{1}a_{2} \rangle, \;
H_{5}=\langle a_{5}, a_{2}a_{3}, a_{2}a_{4} \rangle.$$

The quotient orbifold $S/H_{j}$ has signature $(1;2,2)$ and the underlying Riemann surface is described by the elliptic curve $C_j$ as below.

$$\begin{array}{llll}
C_{1}: & y^{2}=x(x-1)\left(x-\dfrac{\lambda_{2}-1}{\lambda_{1}-1} \right);&
C_{2}: & y^{2}=x(x-1)\left(x-\dfrac{\lambda_{2}}{\lambda_{1}} \right);\\ \\
C_{3}: & y^{2}=x(x-1)\left(x-\dfrac{\lambda_{2}(1-\lambda_{1})}{\lambda_{2}-\lambda_{1}} \right);&
C_{4}: & y^{2}=x(x-1)(x-\lambda_{2});\\ \\
C_{5}: & y^{2}=x(x-1)(x-\lambda_{1}).
\end{array}
$$

\s

In this case, if $i \neq j$, then $H_{i}H_{j}=H_{0}$; so $g_{H_{i}H_{j}}=0$. We may apply Corollary \ref{coroKR} to obtain the following.

\s
\noindent
\begin{theorem}\label{g=5}
Let $S$, $C_{1}$,..., $C_{5}$ be as above. Then
$$JS \cong_{isog.} C_{1} \times C_{2} \times C_{3} \times C_{4} \times C_{5}.$$
\end{theorem}

\s
\noindent
\begin{remark}
\begin{enumerate}
\item Theorem \ref{g=5}, for $\lambda_{1}, \lambda_{2} \in \overline{\mathbb Q}$,  was obtained by Yamauchi in \cite{Yamauchi1}. Our results states that all generalized Fermat curves of type $(2,4)$, a two-dimensional family of Riemann surfaces of genus five, have a complete decomposable Jacobian variety.

\item The elliptic curves $C_{1}$, $C_{2}$ and $C_{3}$, respectively, can also be described by the following ones
$$\widehat{C}_{1}: y^{2}=(x-1)(x-\lambda_{1})(x-\lambda_{2}),$$
$$\widehat{C}_{2}: y^{2}=x(x-\lambda_{1})(x-\lambda_{2}),$$
$$\widehat{C}_{3}: y^{2}=x(x-1)(x-\lambda_{1})(x-\lambda_{2}).$$

If we consider the fiber product of $\widehat{C}_{1}$, $\widehat{C}_{2}$, $\widehat{C}_{3}$, $C_{4}$ and $C_{5}$, in each case using the projection $(x,y) \mapsto x$, we obtain a reducible curve. This curve has two irreducible components, both of them isomorphic to $S$.
\end{enumerate}

\end{remark}

\s
\subsection{An example}
Let $j$ be the Klein-elliptic function
$$j(\lambda)=\frac{(1-\lambda+\lambda^{2})^{3}}{\lambda^{2}(1-\lambda)^{2}}.$$
If we choose $\lambda_{2}=1/\lambda_{1}$, then we see that
$$j\left( \frac{\lambda_{2}(1-\lambda_{1})}{\lambda_{2}-\lambda_{1}} \right) =j(\lambda_{1})$$
is equivalent to have
$$(1+\lambda_{1}^{2})(\lambda_{1}^{2}-\lambda_{1}-1)(\lambda_{1}^{2}+\lambda_{1}-1)=0.$$

$$j\left( \frac{\lambda_{2}}{\lambda_{1}} \right) =j(\lambda_{1})$$
is equivalent to have
$$(\lambda_{1}^{2}-\lambda_{1}-1)(\lambda_{1}^{2}+\lambda_{1}-1)(\lambda_{1}^{2}+\lambda_{1}+1)(\lambda_{1}^{3}-\lambda_{1}+1)(\lambda_{1}^{3}-\lambda_{1}^{2}+1)=0.$$

$$j\left( \frac{\lambda_{2}-1}{\lambda_{1}-1} \right) =j(\lambda_{1})$$
is equivalent to have
$$(1+\lambda_{1}^{2})(\lambda_{1}^{2}-\lambda_{1}-1)(\lambda_{1}^{2}+\lambda_{1}-1)=0.$$

So, if we take $\lambda_{1}=(1-\sqrt{5})/2$ and $\lambda_{2}=-(1+\sqrt{5})/2$, we obtain that 
$C_{1}$, $C_{2}$, $C_{3}$, $C_{4}$ and $C_{5}$ are isomorphic elliptic curves and, in particular, 
$JS \cong_{isog.} C_{1}^{5}$.

\s
\subsection{Generalized Fermat curves of type $(2,5)$}
In this case
\begin{equation}
S=\left\{ \begin{array}{ccc}
x_{1}^{2}+x_{2}^{2}+x_{3}^{2}&=&0\\
\lambda_{1}x_{1}^{2}+x_{2}^{2}+x_{4}^{2}&=&0\\
\lambda_{2}x_{1}^{2}+x_{2}^{2}+x_{5}^{2}&=&0\\
\lambda_{3}x_{1}^{2}+x_{2}^{2}+x_{6}^{2}&=&0\\
\end{array}
\right\} \subset {\mathbb P}^{5},
\end{equation}
where $\lambda_{1}, \lambda_{2}, \lambda_{3} \in {\mathbb C}-\{0,1\}$ and $\lambda_{i} \neq \lambda_{j}$, for $i \neq j$, has genus $g=17$ and the generalized Fermat group $H_{0} \cong {\mathbb Z}_{2}^{5}$ is generated by the transformations 
$$a_{1}([x_{1}:x_{2}:x_{3}:x_{4}:x_{5}:x_{6}])=[-x_{1}:x_{2}:x_{3}:x_{4}:x_{5}:x_{6}],$$ 
$$a_{2}([x_{1}:x_{2}:x_{3}:x_{4}:x_{5}:x_{6}])=[x_{1}:-x_{2}:x_{3}:x_{4}:x_{5}:x_{6}],$$ 
$$a_{3}([x_{1}:x_{2}:x_{3}:x_{4}:x_{5}:x_{6}])=[x_{1}:x_{2}:-x_{3}:x_{4}:x_{5}:x_{6}],$$ 
$$a_{4}([x_{1}:x_{2}:x_{3}:x_{4}:x_{5}:x_{6}])=[x_{1}:x_{2}:x_{3}:-x_{4}:x_{5}:x_{6}],$$ 
$$a_{5}([x_{1}:x_{2}:x_{3}:x_{4}:x_{5}:x_{6}])=[x_{1}:x_{2}:x_{3}:x_{4}:-x_{5}:x_{6}].$$ 

The transformation $a_{6}=a_{1}a_{2}a_{3}a_{4}a_{5}$ is 
$$a_{6}([x_{1}:x_{2}:x_{3}:x_{4}:x_{5}:x_{6}])=[x_{1}:x_{2}:x_{3}:x_{4}:x_{5}:-x_{6}].$$ 

The list of the different subgroups $H$ of $H_{0}$, isomorphic to ${\mathbb Z}_{2}^{4}$ with $S/H$ of genus at least one is the following:
$$H_{1}=\langle a_{1}, a_{2}, a_{3}a_{4}, a_{3}a_{5} \rangle, \; 
H_{2}=\langle a_{1}, a_{3}, a_{2}a_{4}, a_{2}a_{5}  \rangle, \;
H_{3}=\langle a_{1}, a_{4}, a_{2}a_{3}, a_{3}a_{5}  \rangle,$$
$$H_{4}=\langle a_{1}, a_{5}, a_{3}a_{4}, a_{2}a_{3}  \rangle, \;
H_{5}=\langle a_{1}, a_{6}, a_{3}a_{4}, a_{3}a_{5}  \rangle, \;
H_{6}=\langle a_{2}, a_{3}, a_{1}a_{4}, a_{1}a_{5} \rangle,$$
$$H_{7}=\langle a_{2}, a_{4}, a_{1}a_{3}, a_{1}a_{5}  \rangle, \; 
H_{8}=\langle a_{2}, a_{5}, a_{1}a_{4}, a_{1}a_{3}  \rangle, \;
H_{9}=\langle a_{2}, a_{6}, a_{1}a_{4}, a_{1}a_{5}  \rangle,$$
$$H_{10}=\langle a_{3}, a_{4}, a_{1}a_{2}, a_{1}a_{5}  \rangle, \;
H_{11}=\langle a_{3}, a_{5}, a_{1}a_{2}, a_{1}a_{4}  \rangle, \;
H_{12}=\langle a_{3}, a_{6}, a_{1}a_{2}, a_{1}a_{5}  \rangle,$$
$$H_{13}=\langle a_{4}, a_{5}, a_{1}a_{2}, a_{1}a_{3}  \rangle, \;
H_{14}=\langle a_{4}, a_{6}, a_{1}a_{2}, a_{1}a_{3}  \rangle, \;
H_{15}=\langle a_{5}, a_{6}, a_{1}a_{2}, a_{1}a_{3}  \rangle,$$
$$H_{16}=\langle a_{1}a_{2}, a_{1}a_{3}, a_{1}a_{4}, a_{1}a_{5}\rangle.$$

The quotient orbifold $S/H_{j}$, for $j=1,...,15$ has signature $(1;2,2,2,2)$ and its underlying genus one Riemann surface is described by the elliptic curve $C_j$ as below.
$$
\begin{array}{llll}
C_{1}: & y^{2}=x(x-1)\left(x - \dfrac{(\lambda_{3}-\lambda_{2})(1-\lambda_{1})}{(\lambda_{3}-\lambda_{1})(1-\lambda_{2})}\right); &
C_{2}: & y^{2}=x(x-1)\left(x - \dfrac{\lambda_{3}(\lambda_{2}-1)}{\lambda_{2}(\lambda_{3}-\lambda_{1})}\right);\\
C_{3}: & y^{2}=x(x-1)\left(x - \dfrac{\lambda_{3}(1-\lambda_{2})}{\lambda_{3}-\lambda_{2}}\right);&
C_{4}: & y^{2}=x(x-1)\left(x - \dfrac{\lambda_{3}(1-\lambda_{1})}{\lambda_{3}-\lambda_{1}}\right);\\ \\
C_{5}: & y^{2}=x(x-1)\left(x - \dfrac{\lambda_{2}(1-\lambda_{1})}{\lambda_{2}-\lambda_{1}}\right);&
C_{6}: & y^{2}=x(x-1)\left(x - \dfrac{\lambda_{3}-\lambda_{1}}{\lambda_{2}-\lambda_{1}}\right); \\ \\
C_{7}: & y^{2}=x(x-1)\left(x - \dfrac{\lambda_{3}-\lambda_{2}}{1-\lambda_{2}}\right); &
C_{8}: & y^{2}=x(x-1)\left(x - \dfrac{\lambda_{3}-\lambda_{1}}{1-\lambda_{1}}\right);\\ \\
C_{9}: & y^{2}=x(x-1)\left(x - \dfrac{\lambda_{2}-\lambda_{1}}{1-\lambda_{1}}\right); &
C_{10}: & y^{2}=x(x-1)\left(x - \dfrac{\lambda_{3}}{\lambda_{2}}\right);\\ \\
C_{11}: & y^{2}=x(x-1)\left(x - \dfrac{\lambda_{3}}{\lambda_{1}}\right);&
C_{12}: & y^{2}=x(x-1)\left(x - \dfrac{\lambda_{2}}{\lambda_{1}}\right);\\ \\
C_{13}: & y^{2}=x(x-1)\left(x - \lambda_{3}\right); &
C_{14}: & y^{2}=x(x-1)\left(x - \lambda_{2}\right); \\ \\
C_{15}: & y^{2}=x(x-1)\left(x - \lambda_{1}\right).
\end{array}
$$

\s

The quotient Riemann surface $S/H_{16}$ has genus two an it is given by the curve
$$C_{16}: y^{2}=x(x-1)(x-\lambda_{1})(x-\lambda_{2})(x-\lambda_{3}).$$

In this case, if $i \neq j$, then $H_{i}H_{j}=H_{0}$; so $g_{H_{i}H_{j}}=0$. We may apply Corollary \ref{coroKR} to obtain that

\s
\noindent
\begin{theorem}\label{g=17} 
Let $S$, $C_{1}$,..., $C_{16}$ be as above. Then
$$JS \cong_{isog.} C_{1} \times \cdots C_{15} \times JC_{16}.$$
\end{theorem}

\s
\noindent
\begin{remark}
\begin{enumerate}
\item A particular situation of the above result is when 
$\lambda_{1}=-1$, $\lambda_{2} \in {\mathbb C}-\{0,\pm 1, \pm i \}$ and $\lambda_{3}=-1/\lambda_{2}$. In this case the hyperelliptic curve $C_{16}$ admits an involution different from the hyperelliptic involution, this being
$(x,y) \mapsto (-1/x, iy/x^{3})$. By Corollary \ref{hiper2}) $JC_{16} \cong_{isog.} E_{1} \times E_{2}$, where $E_{1}$ and $E_{2}$ are elliptic curves, so $$JS \cong_{isog.} C_{1} \times \cdots C_{15} \times E_{1} \times E_{2}.$$

\item In the case that $\lambda_{1}, \lambda_{2}, \lambda_{3} \in \overline{\mathbb Q}$, Theorem \ref{g=17} was obtained by Yamauchi in \cite{Yamauchi1}. 
\end{enumerate}
\end{remark}

\s

\subsection{Generalized Fermat curves of type $(2,n)$, where $n \geq 6$}
In this case,  
\begin{equation}
S=\left\{ \begin{array}{ccc}
x_{1}^{2}+x_{2}^{2}+x_{3}^{2}&=&0\\
\lambda_{1}x_{1}^{2}+x_{2}^{2}+x_{4}^{2}&=&0\\
\lambda_{2}x_{1}^{2}+x_{2}^{2}+x_{5}^{2}&=&0\\
\vdots & \vdots& \vdots\\
\lambda_{n-2}x_{1}^{2}+x_{2}^{2}+x_{n+1}^{2}&=&0
\end{array}
\right\} \subset {\mathbb P}^{n},
\end{equation}
where $\lambda_{1},...,\lambda_{n-2} \in {\mathbb C}-\{0,1\}$ and $\lambda_{i} \neq \lambda_{j}$, for $i \neq j$, has genus $g=1+2^{n-2}(n-3)$ and the generalized Fermat group  $H_{0} \cong {\mathbb Z}_{2}^{n}$ is generated by the transformations 
$$a_{j}([x_{1}:\cdots:x_{n+1}])=[x_{1}:\cdots:x_{j-1}: -x_{j}:x_{j+1}:\cdots:x_{n+1}], \; j=1,...,n.$$ 

The transformation $a_{n+1}=a_{1}a_{2} \cdots a_{n}$ is
$$a_{n+1}([x_{1}:\cdots:x_{n+1}])=[x_{1}:\cdots:x_{n}:-x_{n+1}].$$

Set 
$$b_{1}=\infty, \; b_{2}=0, \; b_{3}=1, \; b_{4}=\lambda_{1}, \ldots, b_{n+1}=\lambda_{n-2}.$$

\subsubsection{Case $n \geq 6$ even} 
For each $j=1,2,\ldots,(n-2)/2$, we consider the subgroup
$$K_{j}=\langle a_{1}, a_{2},\ldots, a_{2j-1}, a_{2j}a_{2j+1}, a_{2j}a_{2j+2},\ldots, a_{2j}a_{n}\rangle \cong {\mathbb Z}_{2}^{n-1}$$

The quotient orbifold $S/K_{j}$ has signature $((n-2j)/2;2,\stackrel{2(2j-1)}{\ldots},2)$ and admits an automorphism $\tau_{j}$ of order two. If $S_{j}$ is the Riemann surface structure of $S/K_{j}$, then $S_{j}/\langle \tau_{j}\rangle$ is a Riemann orbifold of signature $(0;2,\stackrel{n-2j+2}{\ldots},2)$. The cone points are given by the points $b_{2j},b_{2j+1},\ldots,b_{n+1}$, that is, it is the hyperelliptic Riemann surface of genus $(n-2j)/2$ described by the equation
$$S_{j}: y^{2}=(x-b_{2j})(x-b_{2j+1})\cdots(x-b_{n+1}).$$

Fix $j \in \{1,2,\ldots,(n-2)/2\}$.
For each tuple $(i_{1},\ldots,i_{2j-1}) \in \{1,2,\ldots,n+1\}$, $i_{1}<i_{2}<\cdots<i_{2j-1}$ (we denote by $I_{j}$ such a collection of tuples), we may consider the permutation
$$\sigma_{(i_{1},\ldots,i_{2j-1})}=(1,i_{1})(2,i_{2})\cdots(2j-1,i_{2j-1}) \in \Sigma_{n+1}.$$

We have $\binom{n+1}{2j-1}$ permutations as above. For each such permutation we consider the group
$$H_{\sigma_{(i_{1},\ldots,i_{2j-1})}}$$
$$||$$
$$\langle a_{i_{1}}, a_{i_{2}},\ldots, a_{i_{2j-1}}, a_{\sigma_{(i_{1},\ldots,i_{2j-1})}(2j)}a_{\sigma_{(i_{1},\ldots,i_{2j-1})}(2j+1)}, \ldots, a_{\sigma_{(i_{1},\ldots,i_{2j-1})}(2j)}a_{\sigma_{(i_{1},\ldots,i_{2j-1})}(n)}\rangle \cong {\mathbb Z}_{2}^{n-1}$$

As was the case for $K_{j}$, the quotient orbifold $S/H_{\sigma_{(i_{1},\ldots,i_{2j-1})}}$ has signature $((n-2j)/2;2,\stackrel{2(2j-1)}{\ldots},2)$ and admits an automorphism $\tau_{(i_{1},\ldots,i_{2j-1})}$ of order two. If $S_{(i_{1},\ldots,i_{2j-1})}$ is the Riemann surface structure of $S/H_{\sigma_{(i_{1},\ldots,i_{2j-1})}}$, then $S_{(i_{1},\ldots,i_{2j-1})}/\langle \tau_{(i_{1},\ldots,i_{2j-1})}\rangle$ is a Riemann orbifold of signature $(0;2,\stackrel{n-2j+2}{\ldots},2)$. The cone points are given by the points in 
$$\{c_{1},\ldots,c_{n-2j+2}\}=\{b_{1},\ldots,b_{n+1}\}-\{b_{i_{1}},b_{i_{2}},\ldots,b_{i_{2j-1}}\},$$
that is, it is the hyperelliptic Riemann surface of genus $(n-2j-2)/2$ described by the equation
$$S_{(i_{1},\ldots,i_{2j-1})}: y^{2}=(x-c_{1})(x-c_{2})\cdots(x-c_{n-2j+2}).$$

\s

Any two different subgroups as above, say $H_{\sigma_{(i_{1},\ldots,i_{2j-1})}}$ and $H_{\sigma_{(i_{1},\ldots,i_{2r-1})}}$, satisfies that  
$$H_{\sigma_{(i_{1},\ldots,i_{2j-1})}} H_{\sigma_{(i_{1},\ldots,i_{2r-1})}}=H_{0},$$ 
and, by Lemma \ref{conteo},
$$\sum_{j=1}^{(n-2)/2}\binom{n+1}{2j-1}(n-2j)/2=\sum_{l=2}^{n/2}\binom{n+1}{n+1-2l}(l-1)=\sum_{l=2}^{n/2}\binom{n+1}{2l}(l-1)=1+(n-3)2^{n-2}=g.$$

So, we may apply Corollary \ref{coroKR} to obtain the following.

\s
\noindent
\begin{theorem}
$$JS \cong_{isog.} \prod_{j=1}^{(n-2)/2} \prod_{(i_{1},\ldots,i_{2j-1}) \in I_{j}} JS_{(i_{1},\ldots,i_{2j-1})}$$
\end{theorem}

\s

\subsubsection{Case $n \geq 7$ odd}
For each $j=0,1,\ldots,(n-3)/2$, we consider the subgroup
$$K_{j}=\langle a_{1}, a_{2},\ldots, a_{2j}, a_{2j+1}a_{2j+2}, a_{2j+1}a_{2j+3},\ldots, a_{2j+1}a_{n}\rangle \cong {\mathbb Z}_{2}^{n-1}$$

The quotient orbifold $S/K_{j}$ has signature $((n-2j-1)/2;2,\stackrel{4j}{\ldots},2)$ and admits an automorphism $\tau_{j}$ of order two. If $S_{j}$ is the Riemann surface structure of $S/K_{j}$, then $S_{j}/\langle \tau_{j}\rangle$ is a Riemann orbifold of signature $(0;2,\stackrel{n+1-2j}{\ldots},2)$. The cone points are given by the points $b_{2j+1},\ldots,b_{n+1}$, that is, it is the hyperelliptic Riemann surface of genus $(n-2j-1)/2$ described by the equation
$$S_{j}: y^{2}=(x-b_{2j+1})(x-b_{2j+2})\cdots(x-b_{n+1}).$$

Fix $j \in \{0,1,\ldots,(n-3)/2\}$.
For each tuple $(i_{1},\ldots,i_{2j}) \in \{1,2,\ldots,n+1\}$, $i_{1}<i_{2}<\cdots<i_{2j}$ (we denote by $I_{j}$ such a collection of tuples), we may consider the permutation
$$\sigma_{(i_{1},\ldots,i_{2j})}=(1,i_{1})(2,i_{2})\cdots(2j,i_{2j}) \in \Sigma_{n+1}.$$

We have $\binom{n+1}{2j}$ permutations as above. For each such permutation we consider the group
$$H_{\sigma_{(i_{1},\ldots,i_{2j})}}$$
$$||$$
$$\langle a_{i_{1}}, a_{i_{2}},\ldots, a_{i_{2j}}, a_{\sigma_{(i_{1},\ldots,i_{2j})}(2j+1)}a_{\sigma_{(i_{1},\ldots,i_{2j})}(2j+2)}, \ldots, a_{\sigma_{(i_{1},\ldots,i_{2j})}(2j+1)}a_{\sigma_{(i_{1},\ldots,i_{2j})}(n)}\rangle \cong {\mathbb Z}_{2}^{n-1}$$

As was the case for $K_{j}$, the quotient orbifold $S/H_{\sigma_{(i_{1},\ldots,i_{2j})}}$ has signature $((n-2j-1)/2;2,\stackrel{4j}{\ldots},2)$ and admits an automorphism $\tau_{(i_{1},\ldots,i_{2j})}$ of order two. If $S_{(i_{1},\ldots,i_{2j})}$ is the Riemann surface structure of $S/H_{\sigma_{(i_{1},\ldots,i_{2j})}}$, then $S_{(i_{1},\ldots,i_{2j})}/\langle \tau_{(i_{1},\ldots,i_{2j})}\rangle$ is a Riemann orbifold of signature $(0;2,\stackrel{n+1-2j}{\ldots},2)$. The cone points are given by the points in 
$$\{c_{1},\ldots,c_{n+1-2j}\}=\{b_{1},\ldots,b_{n+1}\}-\{b_{i_{1}},b_{i_{2}},\ldots,b_{i_{2j}}\},$$
that is, it is the hyperelliptic Riemann surface of genus $(n-2j-1)/2$ described by the equation
$$S_{(i_{1},\ldots,i_{2j})}: y^{2}=(x-c_{1})(x-c_{2})\cdots(x-c_{n+1-2j}).$$

\s

Any two different subgroups as above, say $H_{\sigma_{(i_{1},\ldots,i_{2j})}}$ and $H_{\sigma_{(i_{1},\ldots,i_{2r})}}$, satisfies that  $$H_{\sigma_{(i_{1},\ldots,i_{2j})}} H_{\sigma_{(i_{1},\ldots,i_{2r})}}=H_{0},$$
and, by Lemma \ref{conteo},
$$\sum_{j=0}^{(n-3)/2}\binom{n+1}{2j}(n-1-2j)/2=\sum_{l=2}^{(n+1)/2}\binom{n+1}{n+1-2l}(l-1)=\sum_{l=2}^{(n+1)/2}\binom{n+1}{2l}(l-1)=1+(n-3)2^{n-2}=g.$$

So, we may apply Corollary \ref{coroKR} to obtain the following.

\s
\noindent
\begin{theorem}
$$JS \cong_{isog.} \prod_{j=0}^{(n-3)/2} \prod_{(i_{1},\ldots,i_{2j}) \in I_{j}} JS_{(i_{1},\ldots,i_{2j})}$$
\end{theorem}

\s
\noindent
\begin{example}[A generalized Fermat curve of type $(2,6)$ whose Jacobian variety is isogenous to a product of elliptic curves]
We have seen that the Jacobian variety of a generalized Fermat curve of type $(2,4)$, which has genus five, is isogenous to the product of five elliptic curves. Similarly, we have constructed a one-dimensional family of generalized Fermat curves of type $(2,5)$, which have genus seventeen, whose Jacobian variety is isogenous to the product of seventeen elliptic curves. If we consider a generalized Fermat curve of type $(2,6)$, which has genus $49$, the above decomposition states that its Jacobian variety is isogenous to the product of $35$ elliptic curves and $7$ Jacobian varieties of genus two curves. These seven curves of genus two are of the form $y^{2}=\prod_{j=1}^{6}(x-a_{j})$, where $a_{1},\ldots,a_{6} \in \{\infty,0,1,\lambda_{1},\lambda_{2},\lambda_{3},\lambda_{4}\}$. If we assume the values $\lambda_{1}$, $\lambda_{2}$, $\lambda_{3}$ and $\lambda_{4}$ so that there is a M\"obius transformation $T$ of order seven preserving the set $\{\infty,0,1,\lambda_{1},\lambda_{2},\lambda_{3},\lambda_{4}\}$, then for every point $a \in \{\infty,0,1,\lambda_{1},\lambda_{2},\lambda_{3},\lambda_{4}\}$ there is a M\"obius transformation $S_{a}$ of order two fixing $a$ and permuting the other six points. This asserts that the corresponding genus two surface admits a conformal involution with two fixed points, so its Jacobian variety is isogenous to the product of two elliptic curves. In particular, for such values the corresponding generalized Fermat curve of type $(2,6)$ is isogenous to the product of $49$ elliptic curves.

\end{example}

\s

\section{Jacobian variety of generalized Fermat curves of type $(3,n)$, where $n \in \{2,3,4\}$}
In this section we describe the decomposition of the Jacobian variety of generalized Fermat curves of type $(3,n)$, with $n \in \{2,3,4\}$, indicated in Theorem \ref{(p,n)}.

\subsection{A simple remark}
Let us start with the following fact (which was also obtained in \cite{Riera-Rodriguez}).

\s
\noindent
\begin{lemma}\label{lemag=2}
If $S$ is a closed Riemann surface of genus two admitting a conformal automorphism $\eta$ of order three with exactly four fixed points, then $JS$ is isogenous to the product of two elliptic curves.
\end{lemma}
\begin{proof}
By the Riemann-Hurwitz formula, $S/\langle \eta \rangle$ has signature $(0;3,3,3,3)$.  We may assume that the cone points of $S/\langle \eta \rangle$ are $\infty$, $0$, $1$ and $\lambda \in {\mathbb C}-\{0,1\}$. It follows that $S$ is defined by the hyperelliptic curve 
$$y^{2}=(x^{3}-1)\left(x^{3}-\alpha^{3} \right),$$
where
$$\alpha^{3}= \left(\frac{\sqrt{\lambda}+1}{\sqrt{\lambda}-1}\right)^{2}.$$

The automorphism $\eta$ is defined by
$$\eta(x,y)=(\omega_{3}x,y), \quad \omega_{3}=e^{2\pi i/3}.$$

The above hyperelliptic curve admits the involution 
$$(x,y) \mapsto \left(\frac{\alpha}{x}, \frac{\alpha^{3/2} \; y}{x^{3}}\right),$$
with exactly two fixed points; so it follows (see Section \ref{genustwo}) that $JS \cong_{isog.} E_{1} \times E_{2}$, where $E_{1}$ and $E_{2}$ are elliptic curves.

\end{proof}

\s
\subsection{Generalized Fermat curves of type $(3,2)$}
There is exactly one generalized Fermat curve of type $(3,2)$, this being given by:
$$x_{1}^{3}+x_{2}^{3}+x_{3}^{3}=0$$
which has genus $g=1$ (it is also described by the curve $y^{2}=x^{3}-1$).

\s
\subsection{Generalized Fermat curves of type $(3,3)$}
A generalized Fermat curve of type $(3,3)$ is given by 
\begin{equation}
S=\left\{ \begin{array}{ccc}
x_{1}^{3}+x_{2}^{3}+x_{3}^{3}&=&0\\
\lambda x_{1}^{3}+x_{2}^{3}+x_{4}^{3}&=&0
\end{array}
\right\} \subset {\mathbb P}^{3},
\end{equation}
where $\lambda \in {\mathbb C}-\{0,1\}$. 

The Riemann surface $S$ has genus $g=10$ and the generalized Fermat group $H_{0} \cong {\mathbb Z}_{3}^{3}$ is generated by the transformations 
$$a_{1}([x_{1}:x_{2}:x_{3}:x_{4}])=[\omega_{3}x_{1}:x_{2}:x_{3}:x_{4}],$$ 
$$a_{2}([x_{1}:x_{2}:x_{3}:x_{4}])=[x_{1}:\omega_{3}x_{2}:x_{3}:x_{4}],$$ 
$$a_{3}([x_{1}:x_{2}:x_{3}:x_{4}])=[x_{1}:x_{2}:\omega_{3}x_{3}:x_{4}],$$ 
where $\omega_{3}=e^{2 \pi i/3}$.

The transformation $a_{4}^{-1}=a_{1}a_{2}a_{3}$ is
$$a_{4}([x_{1}:x_{2}:x_{3}:x_{4}])=[x_{1}:x_{2}:x_{3}:\omega_{3}x_{4}].$$ 

\s

The list of the different subgroups $H$ of $H_{0}$, isomorphic to ${\mathbb Z}_{3}^{2}$ with $S/H$ of genus at least one is the following:
$$H_{1}=\langle a_{1}, a_{2}a_{3}^{-1} \rangle, \; 
H_{2}=\langle a_{2}, a_{1}a_{3}^{-1} \rangle, \;
H_{3}=\langle a_{3}, a_{2}a_{1}^{-1} \rangle, \;
H_{4}=\langle a_{4}, a_{2}a_{3}^{-1} \rangle,$$
$$H_{5}=\langle a_{1}a_{2}, a_{1}a_{3} \rangle, \;
H_{6}=\langle a_{1}a_{2}, a_{2}a_{3} \rangle, \;
H_{7}=\langle a_{2}a_{3}, a_{1}a_{3} \rangle.$$

If $j=1,2,3,4$, then the quotient orbifold $S/H_{j}$ has signature $(1;3,3,3)$ and the underlying Riemann surface is isomorphic to the elliptic curve
$$y^{2}=x^{3}-1.$$

Note that $S/\langle a_{j} \rangle$ has signature $(1;3,\stackrel{9}{\ldots},3)$ and its 
underlying Riemann surface is isomorphic to the generalized Fermat curve of type $(3,2)$ (which has genus one)
$$C: y_{1}^{3}+y_{2}^{3}+y_{3}^{3}=0.$$

The quotient $S_{j}=S/H_{j}$, for $j=5,6,7$, is a closed Riemann surface of genus two admitting an automorphism $\eta_{j}$ of order three so that  $S_{j}/\langle \eta_{j} \rangle$ has signature $(0;3,3,3,3)$ and whose cone points are $\infty$, $0$, $1$ and $\lambda$. It follows, from Lemma \ref{lemag=2}, that $S_{j}$ is isomorphic to 
$$E: y^{2}=(x^{3}-1)\left(x^{3}-\alpha^{3} \right),$$
where
$$\alpha^{3}= \left(\frac{\sqrt{\lambda}+1}{\sqrt{\lambda}-1}\right)^{2},$$
and that $JE \cong_{isog.} E_{1} \times E_{2}$, where $E_{1}$ and $E_{2}$ are elliptic curves.

We must note that, for $i \neq j$, the group $H_{i}H_{j}$ always contains two elements $a_{u}$ and $a_{v}$, where $u \neq v$. Now, since $S/\langle a_{u},a_{v}\rangle$ has signature $(0;3,\stackrel{6}{\ldots},3)$, for $u \neq v$, one has that $g_{H_{i}H_{j}}=0$ for $i \neq j$. 
We may apply Corollary \ref{coroKR} to obtain the following.

\s
\noindent
\begin{theorem}\label{g=10}
Let $S$ be a generalized Fermat curve of type $(3,3)$. Then $JS$ is isogenous to the product of elliptic curves. \end{theorem}

\s
\subsection{Generalized Fermat curves of type $(3,4)$}
In this case,  
\begin{equation}
S=\left\{ \begin{array}{ccc}
x_{1}^{3}+x_{2}^{3}+x_{3}^{3}&=&0\\
\lambda_{1} x_{1}^{3}+x_{2}^{3}+x_{4}^{3}&=&0\\
\lambda_{2} x_{1}^{3}+x_{2}^{3}+x_{5}^{3}&=&0\\
\end{array}
\right\} \subset {\mathbb P}^{4},
\end{equation}
where $\lambda_{1}, \lambda_{2} \in {\mathbb C}-\{0,1\}$ and $\lambda_{1} \neq \lambda_{2}$,  
has genus $g=55$ and the generalized Fermat group $H_{0} \cong {\mathbb Z}_{3}^{4}$ is generated by the transformations 
$$a_{1}([x_{1}:x_{2}:x_{3}:x_{4}:x_{5}])=[\omega_{3}x_{1}:x_{2}:x_{3}:x_{4}:x_{5}],$$ 
$$a_{2}([x_{1}:x_{2}:x_{3}:x_{4}:x_{5}])=[x_{1}:\omega_{3}x_{2}:x_{3}:x_{4}:x_{5}],$$ 
$$a_{3}([x_{1}:x_{2}:x_{3}:x_{4}:x_{5}])=[x_{1}:x_{2}:\omega_{3}x_{3}:x_{4}:x_{5}],$$ 
$$a_{4}([x_{1}:x_{2}:x_{3}:x_{4}:x_{5}])=[x_{1}:x_{2}:x_{3}:\omega_{3}x_{4}:x_{5}],$$
where $\omega_{3}=e^{2 \pi i/3}$.

The transformation $a_{5}^{-1}=a_{1}a_{2}a_{3}a_{4}$ is 
$$a_{5}([x_{1}:x_{2}:x_{3}:x_{4}:x_{5}])=[x_{1}:x_{2}:x_{3}:x_{4}:\omega_{3}x_{5}].$$

The list of the different subgroups $H$ of $H_{0}$, isomorphic to ${\mathbb Z}_{3}^{3}$ with $S/H$ of genus at least one is the following:

$$H_{1}=\langle a_{1}, a_{2}, a_{3}a_{4}^{-1} \rangle, \; 
H_{2}=\langle a_{1}, a_{3}, a_{2}a_{4}^{-1} \rangle, \; 
H_{3}=\langle a_{1}, a_{4}, a_{2}a_{3}^{-1}\rangle,$$
$$H_{4}=\langle a_{1}, a_{5}, a_{2}a_{3}^{-1}\rangle, \;
H_{5}=\langle a_{2}, a_{3}, a_{1}a_{4}^{-1} \rangle,
H_{6}=\langle a_{2}, a_{4}, a_{1}a_{3}^{-1} \rangle,$$
$$H_{7}=\langle a_{2}, a_{5}, a_{1}a_{3}^{-1}\rangle, \; 
H_{8}=\langle a_{3}, a_{4}, a_{1}a_{2}^{-1}\rangle, \; 
H_{9}=\langle a_{3}, a_{5}, a_{1}a_{2}^{-1}\rangle, \; 
H_{10}=\langle a_{4}, a_{5}, a_{1}a_{2}^{-1}\rangle,$$
$$L_{1}=\langle a_{1}, a_{2}a_{3}, a_{2}a_{4} \rangle, \; L_{2}=\langle a_{1}, a_{2}a_{3}, a_{2}a_{5} \rangle, \; L_{3}=\langle a_{1}, a_{2}a_{4}, a_{2}a_{5} \rangle,$$
$$L_{4}=\langle a_{2}, a_{3}a_{4}, a_{3}a_{5} \rangle,\; L_{5}=\langle a_{2}, a_{3}a_{4}, a_{3}a_{1} \rangle,\; L_{6}=\langle a_{2}, a_{3}a_{5}, a_{3}a_{1} \rangle,$$
$$L_{7}=\langle a_{3}, a_{4}a_{5}, a_{4}a_{1}\rangle,\; L_{8}=\langle a_{3}, a_{4}a_{5}, a_{4}a_{2}\rangle,\; L_{9}=\langle a_{3}, a_{4}a_{1}, a_{4}a_{2}\rangle,$$
$$L_{10}=\langle a_{4}, a_{5}a_{1}, a_{5}a_{2}\rangle,\; L_{11}=\langle a_{4}, a_{5}a_{1}, a_{5}a_{3}\rangle,\; L_{12}=\langle a_{4}, a_{5}a_{2}, a_{5}a_{3}\rangle,$$
$$L_{13}=\langle a_{5}, a_{1}a_{2}, a_{1}a_{3} \rangle,\; L_{14}=\langle a_{5}, a_{1}a_{2}, a_{1}a_{4} \rangle,\; L_{15}=\langle a_{5}, a_{1}a_{3}, a_{1}a_{4} \rangle,$$
$$R_{1}=\langle a_{1}a_{2}, a_{1}a_{3}, a_{1}a_{4} \rangle, \;
R_{2}=\langle a_{1}a_{2}, a_{2}a_{3}, a_{2}a_{4} \rangle,$$
$$R_{3}=\langle a_{3}a_{2}, a_{1}a_{3}, a_{3}a_{4} \rangle, \;
R_{4}=\langle a_{4}a_{2}, a_{4}a_{3}, a_{1}a_{4} \rangle, \;
R_{5}=\langle a_{5}a_{2}, a_{5}a_{3}, a_{5}a_{4} \rangle.$$

\s

\subsubsection{} The quotient orbifold $S/H_{j}$, for each $j=1,\ldots,10$, has signature $(1;3,\stackrel{6}{\ldots},3)$ whose underlying Riemann surface is isomorphic to the elliptic curve
$$y^{2}=x^{3}-1.$$

Since, for every triple $i_{1},i_{2},i_{3} \in \{1,2,3,4,5\}$ of pairwise different numbers one has that $S/\langle a_{i_{1}}, a_{i_{2}}, a_{i_{3}}\rangle$ has signature $(0;3,\stackrel{9}{\ldots},3)$, for $i\neq j$ it holds that $S/H_{i}H_{j}$ has genus zero.

\subsubsection{}
The quotient orbifold $S/L_{j}$ has signature $(2;3,3,3)$ and its underlying Riemann surface $S_{j}$ (of genus two) admits an action of ${\mathbb Z}_{3}$ with quotient orbifold with signature $(0;3,3,3,3)$; so Lemma \ref{lemag=2} states that $JS_{j}$ is isogenous to the product of two elliptic curves. 

It can be checked that $L_{i}L_{j}=H_{0}$ if $i \neq j$, so $S/L_{i}L_{j}$ has genus zero. Also, we may see that $L_{j}H_{i}=H_{0}$, for $j=1,\ldots, 15$ and $i=1,\ldots,10$, so $S/L_{j}H_{i}$ has also genus zero.

\subsubsection{}
The quotient orbihold $S/R_{j}$, in these cases, has signature $(3;-)$. This quotient surface admits a cyclic group of order three, say generated by $\rho_{j}$, as a group of automorphisms with quotient being of signature $(0;3,3,3,3,3)$. In fact, this Riemann surface can be described by the algebraic curve
$$y^{3}=x^{\alpha_{1}}(x-1)^{\alpha_{2}}(x-\lambda_{1})^{\alpha_{3}}(x-\lambda_{2})^{\alpha_{4}},$$
where
$$(\alpha_{1},\alpha_{2},\alpha_{3},\alpha_{4})=\left\{ \begin{array}{cc}
(1,2,2,2), & j=1;\\
(2,1,1,1), & j=2;\\
(2,2,1,2), & j=3;\\
(2,2,2,1), & j=4;\\
(2,2,2,2), & j=5.
\end{array}
\right.
$$
and $\rho_{j}(x,y)=(x,e^{2 \pi i/e}y)$.

For any different two of these groups, say $R_{i}$ and $R_{j}$ one has that $R_{i}R_{j}=H_{0}$, so $S/R_{i}R_{j}$ has genus zero.

Similarly, $S/H_{i}R_{j}$ and $S/L_{i}R_{j}$ have genus zero.

\s
\noindent
\begin{remark}
\begin{enumerate}
\item We explain the difference between the above groups $L_{j}$ above. First, once we have made a choice of one of the points inside $\{\infty,0,1,\lambda_{1},\lambda_{2}\}$, we may consider a closed Riemann surface of genus two admitting an automorphism of order three with quotient being the sphere and whose branch values are the other four fixed points. The choice of a point as above is equivalent to make a choice in the set $\{a_{1},a_{2},a_{3},a_{4},a_{5}\}$. For simplicity, let us assume the chosen point is $\infty$, equivalently, the element $a_{1}$. So the Riemann surface of genus two we look for must be represented by a curve of the form
$$y^{3}=x^{\alpha_{1}}(x-1)^{\alpha_{2}}(x-\lambda_{1})^{\alpha_{3}}(x-\lambda_{2})^{\alpha_{4}},$$
where $\alpha_{1},\alpha_{2},\alpha_{3},\alpha_{4} \in \{1,2\}$ and $\alpha_{1}+\alpha_{2}+\alpha_{3}+\alpha_{4}$ is congruent to $0$ module $3$. For this two happen, two of these values must be equal to $1$ and the other two equal to $2$. This provides $6$ possible ordered tuples $(\alpha_{1},\alpha_{2},\alpha_{3},\alpha_{4})$. But, the tuples $(1,1,2,2)$ and $(2,2,1,1)$ produce isomorphic curves; similarly the pairs of tuples $(1,2,1,2)$ and $(2,1,2,1)$, the pairs of tuples $(1,2,2,1)$ and $(2,1,1,2)$. So, we only have $3$ cases to consider, these provided by the tuples $(1,2,2,1)$, $(1,2,1,2)$ and $(1,1,2,2)$. These three cases correspond, respectively,  to the groups $L_{1}$, $L_{2}$ and $L_{3}$.

\s

\item Let us consider the following subgroups of $H_{0}$, each one isomorphic to ${\mathbb Z}_{3}^{2}$, 
$$U_{1}=\langle a_{1}, a_{2}a_{3} \rangle=\langle a_{1}, a_{4}a_{5} \rangle,\;
U_{2}=\langle a_{1}, a_{2}a_{4} \rangle=\langle a_{1}, a_{3}a_{5} \rangle,\; 
U_{3}=\langle a_{1}, a_{2}a_{5} \rangle=\langle a_{1}, a_{3}a_{4} \rangle,$$
$$U_{4}=\langle a_{2}, a_{1}a_{3}\rangle=\langle a_{2}, a_{4}a_{5}\rangle,\; 
U_{5}=\langle a_{2}, a_{1}a_{4}\rangle=\langle a_{2}, a_{3}a_{5}\rangle,\;
U_{6}=\langle a_{2}, a_{1}a_{5}\rangle=\langle a_{2}, a_{3}a_{4}\rangle, $$
$$U_{7}=\langle a_{3}, a_{1}a_{2} \rangle=\langle a_{3}, a_{4}a_{5} \rangle,\;
U_{8}=\langle a_{3}, a_{1}a_{4} \rangle=\langle a_{3}, a_{2}a_{5} \rangle, \;
U_{9}=\langle a_{3}, a_{1}a_{5} \rangle=\langle a_{3}, a_{2}a_{4} \rangle, $$
$$U_{10}=\langle a_{4}, a_{1}a_{2}\rangle=\langle a_{4}, a_{3}a_{5}\rangle,\;
U_{11}=\langle a_{4}, a_{1}a_{3}\rangle=\langle a_{4}, a_{2}a_{5}\rangle, \;
U_{12}=\langle a_{4}, a_{1}a_{5}\rangle=\langle a_{4}, a_{2}a_{3}\rangle, $$
$$U_{13}=\langle a_{5}, a_{1}a_{2}\rangle=\langle a_{5}, a_{3} a_{4}\rangle, \;
U_{14}=\langle a_{5}, a_{1}a_{3}\rangle=\langle a_{5}, a_{2}a_{4}\rangle,\;
U_{15}=\langle a_{5}, a_{1}a_{4}\rangle=\langle a_{5}, a_{2}a_{3}\rangle.$$

The quotient orbifold $S/U_{j}$, for each $j=1,\ldots,15$, has signature $(4;3,\stackrel{9}{\ldots},3)$. The underlying Riemann surface of genus four, say $S_{j}$, admits a group $K_{j} \cong {\mathbb Z}_{3}^{2}$ of conformal automorphisms with quotient orbifold of signature $(0;3,3,3,3)$. It is well known that $K_{j}$ can be generated by two elements of order three, say $b_{1}$ and $b_{2}$, each acting with fixed points, so that the only non-trivial elements acting with fixed points are $b_{1}$, $b_{1}^{2}$, $b_{2}$ and $b_{2}^{2}$ (for instance, by using Theorem \ref{cor2}). Let $L_{j,1}=\langle b_{1}b_{2} \rangle$ and $L_{j,2}=\langle b_{1}b_{2}^{2}\rangle$. Then, for $i=1,2$, $S_{j}/L_{j,i}$ has signature $(2;-)$ and $L_{j,1}L_{j,2}=K_{j}$. It follows from Corollary \ref{coroKR} that $JS_{j}\cong_{isog.} JS/L_{j,1} \times JS/L_{j,2}$. But $S/L_{j,i}$ is a genus two Riemann surface admitting a conformal automorphism of order three acting with four fixed points, so Lemma \ref{lemag=2} states that $JS/L_{j,i}$ is isogenous to the product of two elliptic curves. As a consequence, $JS_{j}$ is isogenous to the product of four elliptic curves.
If $i_{1} \in \{1,2,3\}$, $i_{2} \in \{4,5,6\}$, $i_{3} \in \{7,8,9\}$, $i_{4} \in \{10,11,12\}$, $i_{5} \in \{13,14,15\}$, then for $u,v \in \{i_{1},i_{2},i_{3},i_{4}\}$, $u \neq v$, one has that $U_{u}U_{v}$ contains at least three of the elements $a_{1}$, $a_{2}$, $a_{3}$, $a_{4}$ and $a_{5}$. So, $S/U_{u}U_{v}$ has genus zero.  
Unfortunately, if $u,v$ belong to the same triple ($\{1,2,3\}$, $\{4,5,6\}$, $\{7,8,9\}$, $\{10,11,12\}$ or $\{13,14,15\}$), then $U_{u}U_{v} \cong {\mathbb Z}_{3}^{3}$ is one of the groups $L_{j}$ above, so the quotient $S/U_{u}U_{v}$ has positive genus.
\end{enumerate}
\end{remark}

\s

As  the total sum for the genera of all orbifolds is
$$\sum_{j=1}^{10}\underbrace{g(S/{H_{j}})}_{1} + \sum_{j=1}^{15}\underbrace{g(S/L_{j})}_{2}+ \sum_{j=1}^{5}\underbrace{g(S/R_{j})}_{3}=55,$$
Corollary \ref{coroKR}  states that

\s
\noindent
\begin{theorem}\label{(3,4)}
If $S$ is a generalized Fermat curve of type $(3,4)$, then 
$$JS \cong_{isog.} \prod_{j=1}^{10} JS_{H_{j}} \prod_{j=1}^{15}JS_{L_{j}} \prod_{j=1}^{5}JS_{R_{j}}.$$

Moreover, each $JS_{H_{j}}$ is an elliptic curve and each $JS_{L_{j}}$ is isogenous to the product of two elliptic curves.

\end{theorem}

\s

\subsection{Elliptic curves decomposition?}
Can we find parameters $\lambda_{1}$ and $\lambda_{2}$ so that each of the $3$-dimensional Jacobians $JS_{R_{j}}$ is isogenous to the product of three elliptic curves?

Let us assume there is a M\"obius transformation $T$ sending  $\{1,\omega_{5}, \omega_{5}^{2}, \omega_{5}^{3}, \omega_{5}^{4}\}$ into $\{\infty,0,1,\lambda_{1}, \lambda_{2}\}$, where $\omega_{5}=e^{2 \pi i/5}$. For instance, if we take $T(x)=(x-\omega_{5}^{4})(1-\omega_{5})/(x-\omega_{5})(1-\omega_{5}^{4})$, then $T(1)=1$, $T(\omega_{5})=\infty$, $T(\omega_{5}^{4})=0$. So, in this case, we may take 
$$\lambda_{1}=T(\omega_{5}^{2})=-\omega_{5}/(1+\omega_{5}^{2})$$
$$\lambda_{2}=T(\omega_{5}^{4})=-\omega_{5}^{2}/(1+\omega_{5}^{2})(1+\omega_{5})^{2}.$$

With the above parameters one may check that the five genus three Riemann surfaces $S_{R_{j}}$ are isomorphic to
$$y^{3}=(x-1)^{2}(x-\omega_{5})(x-\omega_{5}^{2})(x-\omega_{5}^{3})(x-\omega_{5}^{4}),$$
and the oder three automorphism $\rho_{j}$ is given by $\rho(x,y)=\left(x,e^{2 \pi i/e}y\right)$. Now, Theorem \ref{(3,4)} asserts that
$$JS \cong_{isog.} \prod_{j=1}^{10} JS_{H_{j}} \prod_{j=1}^{15}JS_{L_{j}} \times JS_{R_{1}}^{5}.$$

Also, $S$ now admits a conformal involution given by
$$A([x_{1}:x_{2}:x_{3}:x_{4}:x_{5}])=\left[x_{3}: -x_{4}: (1-\lambda_{1})^{1/3} x_{1}:  (\lambda_{1}-1)^{1/3} x_{2}: \left(\frac{1-\lambda_{1}}{1-\lambda_{2}}\right)^{1/3} x_{5} \right].$$

This involution induces a conformal involution $\tau$ in the genus three Riemann surface $X=S/R_{1}$, with quotient $X/\langle \tau \rangle$ of signature $(1;2,2,2,2)$. Also, one of the fixed points of $\tau$ is also a fixed point of the automorphism of order three $\rho_{1}$ obtained previously; so $\rho\tau$ is an automorphism of order six and $X/\langle \rho\tau \rangle$ has signature $(0;2,3,3,6)$. In the above algebraic model, 
$$\tau(x,y)=\left(\frac{1}{x},\frac{y}{x^{2}}\right).$$

In particular, $S/R_{j}$ can also be described by the curve
$$C: y^{6}=x^{3}(x-1)^{2}(x-\mu)^{2}, \quad \mu=\frac{(\omega_{5}^{2}+1)^{2}}{(\omega_{5}+1)^{4}},$$
and the automorphism of order six is given by 
$$\rho\tau(x,y)=\left(x, e^{\pi i/3}y \right).$$

Question: Is $JC$ isogenous to the product of three elliptic curves?

\section{A conjectural picture}
Theorem \ref{(p,n)} states a decomposition for the Jacobian variety of a generalized Fermat curve of type $(p,n)$, where $p$ is a prime integer. What is going on with the case $p$ not a prime?

In the following three examples, in the case of classical Fermat curves, we may see that Theorem \ref{(p,n)} does not hold in general if $p$ is not a prime integer. After these examples we propose a conjectural decomposition result.

\s
\noindent
\begin{example}[Fermat curve $F_{4}$] In this example we observe that Theorem \ref{(p,n)} holds for a generalized Fermat curve of type $(4,2)$, that is, when $S=F_{4}=\{x_{1}^{4}+x_{2}^{4}+x_{3}^{4}=0\} \subset {\mathbb P}^{2}$
is the classical Fermat curve of genus $3$.
If $H<H_{0}=\langle a_{1}, a_{2}\rangle$ is so that $S/H$ has positive genus, then $S/H$ has signature $(1;2,2)$ and $H \cong {\mathbb Z}_{4}$. The collection of such subgroups $H$ 
 is the following:
$$H_{1}=\langle a_{1}a_{2}^{2}\rangle, \;
H_{2}=\langle a_{2}a_{1}^{2}\rangle, \;
H_{3}=\langle a_{3}a_{2}^{2}\rangle.$$

The underlying Riemann surface $S_{H_{j}}$ of genus one admits a conformal automorphism of order four, say $\tau_{j}$, with 
$S_{H_{j}}/\langle \tau_{j} \rangle$ with signature $(0;2,4,4)$. In particular, all of them are isomorphic to the elliptic curve
$$E: y^{2}=x^{4}-1.$$

Since $H_{i}H_{j}=H_{0}$, for $i \neq j$, we may apply Corollary \ref{coroKR} to obtain that
$$JF_{4}\cong_{isog.} E^{3}.$$

\end{example}

\s
\noindent
\begin{example}[Fermat curve $F_{6}$] In this example we observe that Theorem \ref{(p,n)} does not hold for a generalized Fermat curve of type $(6,2)$, that is, when $S=F_{6}=\{x_{1}^{6}+x_{2}^{6}+x_{3}^{6}=0\} \subset {\mathbb P}^{2}$
is the classical Fermat curve of genus $10$.
If $H<H_{0}=\langle a_{1}, a_{2}\rangle$ is so that $S/H$ has positive genus, then either:
\begin{enumerate}
\item $S/H$ has signature $(1;2,2,2)$, $H=H_{1}=\langle a_{1}a_{2}^{-1}, a_{1}^{3}\rangle$ and $S_{H}$ is given by the elliptic curve $E: y^{2}=x^{3}-1$;

\item $S/H$ has signature $(1;2,2,3,3,3)$, $S_{H}$ is also given by $E$, and $H \cong {\mathbb Z}_{6}$ is any of the followings:
$$H_{2}=\langle a_{1}^{3},a_{2}^{2}\rangle, \; 
H_{3}=\langle a_{1}^{3},a_{3}^{2}\rangle, \;
H_{4}=\langle a_{1}^{2},a_{2}^{3}\rangle, \; 
H_{5}=\langle a_{1}^{2},a_{3}^{3}\rangle, \; 
H_{6}=\langle a_{2}^{3},a_{3}^{2}\rangle, \; 
H_{7}= \langle a_{2}^{2},a_{3}^{3}\rangle; 
$$

\item $S/H$ has signature $(2;2,2)$, $S_{H}$ is given by $F: y^{2}=x^{6}-1$, and $H \cong {\mathbb Z}_{6}$ is any of the followings:
$$H_{8}=\langle a_{1}a_{2}^{-1}\rangle, \; 
H_{9}=\langle a_{1}a_{3}^{-1}\rangle, \;
H_{10}=\langle a_{2}a_{3}^{-1}\rangle.
$$

\end{enumerate}

It can be checked that in this case Theorem \ref{(p,n)} fails using the subgroups $H_{1}$,..., $H_{10}$. Anyway, in this case we have that $F_{6}$ admits the symmetric group ${\mathfrak S}_{3}=\langle\tau_{8},\tau_{9},\tau_{10}\rangle$ as a group of conformal automorphisms, where $\tau_{8}$ is involution that conjugates $a_{1}$ to $a_{2}$ and fixes $a_{3}$, $\tau_{9}$ is involution that conjugates $a_{1}$ to $a_{3}$ and fixes $a_{2}$ and $\tau_{10}$ is involution that conjugates $a_{2}$ to $a_{3}$ and fixes $a_{1}$. 
The group $H_{0}$ is normalized by ${\mathfrak S}_{3}$ and, in particular, the three subgroups $H_{8}$, $H_{9}$ and $H_{10}$ are permuted by conjugation of elements in ${\mathfrak S}_{3}$. The stabilizer, in ${\mathfrak S}_{3}$, of $H_{j}$ is $\langle \tau_{j} \rangle$, for $j=8,9,10$.
If we set $K_{8}=\langle H_{8}, \tau_{8}\rangle$, $K_{9}=\langle H_{9},\tau_{9}\rangle$ and $K_{10}=\langle H_{10},\tau_{10}\rangle$, then $S/K_{j}$ is a genus one Riemann surface, the three of them isomorphic to $E$. It follows (using the groups $H_{1}$,..., $H_{7}$, $K_{8}$, $K_{9}$ and $K_{10}$ in Kani-Rosen theorem) that $$JF_{6} \cong_{isog.} E^{10}.$$

The above example was also been obtained by Beauville in \cite{Beauville}.

\end{example}

\s
\noindent
\begin{example}[Fermat curve $F_{8}$] Let us consider the generalized Fermat curve of type $(8,2)$, that is, when $S=F_{8}=\{x_{1}^{8}+x_{2}^{8}+x_{3}^{8}=0\} \subset {\mathbb P}^{2}$
is the classical Fermat curve of genus $21$.
If $H<H_{0}=\langle a_{1}, a_{2}\rangle$ is so that $S/H$ has positive genus and $H_{0}/H \cong {\mathbb Z}_{8}$, then \begin{enumerate}
\item $S/H$ has signature $(2;4,4,4,4)$ and $H$ is one of the following groups:
$$H_{1}=\langle a_{1}^{-2}, a_{1}a_{2}^{4}\rangle, \quad H_{2}=\langle a_{2}^{-2},a_{1}^{4}a_{2}^{-1}\rangle;$$

\item $S/H$ has signature $(3;2,2)$  and $H$ is any of the following:
$$H_{3}=\langle a_{1}^{-1}a_{2}\rangle, \; 
H_{4}=\langle a_{1}^{2}a_{2}^{-1}\rangle, \;
H_{5}=\langle a_{1}a_{2}^{-2}\rangle, $$
$$H_{6}=\langle a_{1}^{-3}a_{2}^{-1}\rangle, \; 
H_{7}=\langle a_{1}^{-3}a_{2}\rangle, \; 
H_{8}= \langle a_{1}a_{2}^{2}\rangle, \; 
H_{9}=\langle  a_{1}^{2}a_{2}  \rangle.
$$

\end{enumerate}

If $j \in \{3,\ldots,9\}$, the quotient $F_{8}/H_{j}$ admits the group $H_{0}/H_{j} \cong {\mathbb Z}_{8}$ as group of conformal automorphisms with quotient having signature $(0;4,8,8)$. It is known that such genus three Riemann surface is given by the hyperelliptic curve
$$E: y^{2}=x^{8}-1.$$

The groups $H_{3}$, $H_{4}$, $H_{5}$, $H_{6}$, $H_{7}$, $H_{8}$ and $H_{9}$ satisfy the conditions of Corollary \ref{coroKR}; so $$JF_{8} \cong_{isog.} JE^{7}.$$

In \cite{Paulhus1} (case $\alpha=0$ in the proof of Theorem 5 for case $D_{8} \times C_{2}$)  it was proved that $JE \cong_{isog.} E_{1}^{2} \times E_{2}$, where 
$$E_{1}: y^{2}=x^{4}+2, \quad E_{2}: v^{2}=u^{4}+1.$$  

Since $E_{1} \cong E_{2}$, by $(x,y)=(\sqrt[4]{2} u, \sqrt{2} v)$, 
$$JF_{8} \cong_{isog.} E_{2}^{21}.$$

\end{example}

\s

The following is a conjectural picture for the general case (see Lemma \ref{conteo}).

\s
\noindent
\begin{conjecture}
Let $k,n \geq 2$ be integers and set $r_{2}=4$ and $r_{k}=3$, for $k \geq 3$. We also assume that $r_{k} \leq n+1$. For each $r\in\{r_{k},\ldots,n+1\}$ we consider the following sets.
$${\mathcal U}_{r,k}=\{(\alpha_{1},\ldots,\alpha_{r}): \alpha_{j} \in {\mathbb Z}_{k}-\{0\}, \quad \alpha_{1}+\cdots+\alpha_{r} \equiv 0 \mod(k)\}/{\mathbb Z}_{k}^{*},$$
where ${\mathbb Z}_{k}^{*}$ denotes the units of ${\mathbb Z}_{k}$ and the action of $u \in {\mathbb Z}_{k}^{*}$ is given by
$$u \cdot (\alpha_{1},\ldots,\alpha_{r})= (u\alpha_{1},\ldots,u\alpha_{r}),$$
and $A_{r}$ is a maximal set of $r$-tuples $(a_{1},\ldots,a_{r}) \in \{\infty,0,1,\lambda_{1},\ldots,\lambda_{n-2}\}^{r}$, $a_{i} \neq a_{j}$, with the property that if $(a_{1},\ldots,a_{r})$ and $(b_{1},\ldots,b_{r})$ are different elements of $A_{r}$, then there is no a M\"obius transformation $T \in {\rm PSL}_{2}({\mathbb C})$ keeping invariant the set 
$\{\infty,0,1,\lambda_{1},\ldots,\lambda_{n-2}\}$ sending the set $\{a_{1},\ldots,a_{r}\}$ onto $\{b_{1},\ldots,b_{r}\}$.

For each $a=(a_{1},\ldots,a_{r}) \in A_{r}$ and each $\alpha=[(\alpha_{1},\ldots,\alpha_{r})] \in {\mathcal U}_{r,k}$ we consider an irreducible component $C_{(a,\alpha)}$ of the (possible reduced) algebraic curve $$E_{(a,\alpha)}: y^{k}=\prod_{j=1}^{r}(x-a_{j})^{\alpha_{j}}$$
and let $A_{(a,\alpha)}=JC_{(a,\alpha)}$.

Then
$$JC^{k}(\lambda_{1},\ldots,\lambda_{n-2}) \cong_{isog.} \prod_{a \in A_{r}, \alpha \in {\mathcal U}_{r,k}, r_{k} \leq r \leq n+1} A_{(a,\alpha)}.$$

\end{conjecture}

\s



\end{document}